\RequirePackage{fix-cm}
\documentclass[smallextended]{svjour3}       % onecolumn (second format)
\usepackage{amsthm}
\usepackage{graphicx}
\usepackage{subfigure}
\usepackage{threeparttable}
\usepackage{amsmath}
\usepackage{color}
\usepackage{amssymb}
\usepackage{algorithmic}
\usepackage{algorithm}
\usepackage{array}
\usepackage{lipsum}
\usepackage{graphicx}
\usepackage{epstopdf}
\usepackage{ragged2e} 
\usepackage{booktabs,makecell, multirow, tabularx}
\usepackage[misc]{ifsym}
\newtheorem{assumption}{Assumption}
\numberwithin{example}{section} 
\begin{document}

\title{Smoothing Iterative Consensus-based Optimization Algorithm for Nonsmooth Nonconvex Optimization Problems with Global Optimality%\thanks{Grants or other notes
}

\titlerunning{SICBO Algorithm for Nonsmooth Nonconvex Optimization Problems}        % if too long for running head

\author{Jiazhen Wei        \and
	Wei Bian\textsuperscript{*} %etc.
}


\institute{Jiazhen Wei \at
	School of Mathematics, Harbin Institute of Technology, Harbin 150001, China \\
	\email{jiazhenwei98@163.com}           %  \\
	%%            \emph{Present address:} of F. Author  %  if needed
	\and
	Wei Bian (Corresponding author)\at
	School of Mathematics, Harbin Institute of Technology, Harbin 150001, China\\
	\email{bianweilvse520@163.com}
}


\date{Received: date / Accepted: date}

\maketitle

\begin{abstract}
In this paper, we focus on finding the global minimizer of a general unconstrained nonsmooth nonconvex optimization problem. Taking advantage of the smoothing method and the consensus-based optimization (CBO) method, we propose 
a novel smoothing iterative consensus-based optimization (SICBO) algorithm. First, we prove that the solution process of the proposed algorithm here exponentially converges to a common stochastic consensus point almost surely. Second,  we establish a detailed theoretical analysis to ensure the small enough error between the objective function value at the consensus point and the optimal function value, to the best of our knowledge, which provides the first theoretical guarantee to the global optimality of the proposed algorithm for nonconvex optimization problems. Moreover,
unlike the previously introduced CBO methods, the theoretical results are valid for the cases that the objective function is nonsmooth, nonconvex and perhaps non-Lipschitz continuous. 
Finally, several numerical examples are performed to illustrate the effectiveness of our proposed algorithm for solving the global minimizer of the nonsmooth and nonconvex optimization problems.
\end{abstract}

\section{Introduction}
In the past several decades, nonsmooth and nonconvex optimization has played an important role in the fields of image processing and machine learning. Even in some minimization models, the objective functions include some non-Lipschitz regularization terms. However, most nonconvex optimization problems are typically NP hard to solve optimally. As a rule, the local minimizers to nonconvex optimization problems are not unique. % due to their nonconvexity. 
Traditional optimization methods for solving such problems only can obtain local minimizers or even stationary points generally, while in many applications, global minimizers or approximate global minimizers are often required. 

Most existing global optimization algorithms can be divided into deterministic methods and stochastic methods. The branch and bound (BB) method \cite{MorrisonBranch2016} is one of the deterministic methods. It has been successfully applied to certain types of global optimization problems, including concave programming, DC programming, etc. Hansen \cite{HansenGlobal1980} introduced an interval algorithm by combining the Moore-Skelboe  algorithm \cite{MooreInterval1966} and the BB algorithm. Its main advantage is the ability to find all global minimizers to optimization problems within a given range of error accuracy. The tunneling method \cite{LevyThe1985} and the filled function method \cite{GeA1990} are classified as the auxiliary function approaches, which mainly study how to jump out of the local minimas or stationary points, and are quite good at local searching. The authors in \cite{PandiyaNon2021} proposed the non parameter-filled function and indicated that it can examine the global minimizer effectively by some experiments. However, these algorithms all have high requirements for the analytical properties of the problems, such as monotonicity, equicontinuity, Lipschitz differentiability and so forth. Furthermore, they can not guarantee to find the global minimizers.
Stochastic methods utilize probabilistic mechanisms rather than deterministic sequences to describe iterative process. Swarm intelligence (SI) algorithms are trustworthy global stochastic optimization methods that are characterized by simple principle and easy realization. Motivated by various group behaviors in nature, the SI algorithm utilizes exploring the search space and exploiting the current knowledge to update iterations. Particle swarm optimization (PSO) \cite{ThangarajParticle2011}, ant colony optimization (ACO) \cite{MohanA2012} and artificial bee colony (ABC) algorithm \cite{KarabogaA2014} are a few examples of well-known SI methods.

Lately, a novel SI model, namely the consensus-based optimization (CBO) algorithm, was introduced to look for a global minimizer $x^\ast$ of the smooth function $f:\mathbb{R}^d\rightarrow\mathbb{R}_+$ in \cite{PinnauA2016}. %It is also a derivative-free algorithm, which requires only the availability of objective function values but no derivative information. 
Let $x^i(t)=(x^i_1(t),\ldots,x^i_d(t))\in\mathbb{R}^d$ be the position of the $i$-th particle at time $t$. $\{W_l(t)\}_{l=1}^d$ is independent standard Brownian motion. $e_l$ is the column vector in $\mathbb{R}^d$ with $1$ at the $l$-th element and $0$ for the others. %if it exists
%\begin{equation}\label{problem}
%	x^\ast\in\arg\min_{x\in\mathbb{R}^d}f(x).
%\end{equation}
Ha et al. \cite{HaConvergence2020} studied the following variant of CBO algorithm:
\begin{equation}\label{cbo}
	\begin{cases}
		\displaystyle \mathrm{d}x^i(t)=-\lambda(x^i(t)-\bar{x}^{\star}(t))\mathrm{d}t+\sigma\sum_{l=1}^d(x_l^i(t)-\bar{x}^{\star}(t))\mathrm{d}W_l(t)e_l,~ i=1,\ldots,N,\\
		\displaystyle \bar{x}^{\star}(t):=\dfrac{\sum_{i=1}^Nx^i(t)\mathrm{e}^{-\beta f(x^i(t))}}{\sum_{i=1}^N\mathrm{e}^{-\beta f(x^i(t))}},
	\end{cases}
\end{equation}
where $\lambda>0$ and $\sigma>0$ control the drift towards $\bar{x}^{\star}(t)$ and the influence of the stochastic component, respectively, and $\beta$ is also a positive constant. 

The CBO algorithm \eqref{cbo} is a gradient-free algorithm, which does not need the gradient of the objective function. It updates the particle's position starting from an initial set of particles $\{x^i(0)\}_{i=1}^N$ to explore the objective function landscape and finally concentrates around the global minimizer. Specifically, in \eqref{cbo}, the first term $-\lambda(x^i(t)-\bar{x}^{\star}(t))\mathrm{d}t$ is a drift term, where particles are attracted and concentrated around $\bar{x}^\star$. And a stochastic search term that is modeled as an independent Brownian motion makes up the second term $\sigma\sum_{l=1}^d(x_l^i(t)-\bar{x}^{\star}(t))\mathrm{d}W_l(t)e_l$. Particles close to $\bar{x}^{\star}$ exhibit lesser randomness, while particles away from $\bar{x}^{\star}$ exhibit more noise in their search paths, which prompts them to seek a greater area. It is notable that the noise here is the same for all particles, which is a simplified version of the system in \cite{CarrilloA2019}. However, it can be still proved that the sufficient condition for the convergence of \eqref{cbo} is dimensional independent. Moreover, the term $\mathrm{e}^{-\beta f(x^i(t))}$ in the definition of $\bar{x}^{\star}(t)$ is the Gibbs distribution corresponding to  objective function $f(x^i(t))$ and temperature $\frac{1}{\beta}$. 
%It is notable that the noise here is heterogeneous, i.e., the particles explore different dimensions at different rate. According to \cite{CarrilloA2019}, heterogeneous noise causes particles' rates of convergence to be independent of their dimensionality, which makes the technique more effective for solving high-dimensional optimization problems.

There exists a considerable number of works exploring the improved CBO algorithms to solve optimization problem faster and more efficiently, e.g. \cite{BorghiConsensus2023,FornasierConsenseus2021,KoConvergence2022,2020Consensus}. Most methods use the corresponding mean-field limit to understand the convergence of the algorithm or lack rigorous theoretical analysis. It is important to note that by letting $N\rightarrow\infty$, particles are considered to behave independently according to the same law, which makes the large system of equations reduce to a single partial differential equation: the so-called mean-field model. However, the convergence study on corresponding mean-field limit does not imply the convergence of the CBO method per se \cite{HaConvergence2020}.  So it was stated in \cite{Carrillo2023} that directly studying the finite particle system without passing to the mean-field limit is of particular interest. We also refer the readers to \cite{TotzeckTrends2022} for an overview of current development.

%Considering the time-discrete analogue of \eqref{cbo}, we set
%$$h:=\Delta t,\quad x^{(k)}:=x(kh),\quad k=0,1,\ldots.$$
%Very recent paper \cite{HaConvergence2021} gave the following three discrete algorithms:
%\begin{equation*}
%	\begin{split}
	%		&\mbox{Model A:}~x^{i,k+1}=x^{(k)}-\lambda h\left( x^{(k)}-\bar{x}^{\star,k}\right) -\sum^d_{l=1}\left( x_l^{i,k}-\bar{x}^{\star,k}_l\right) \sigma \sqrt{h}Z_l^{(k)}e_l,\\
	%		&\mbox{Model B:}~ 
	%		\begin{cases}
		%			\displaystyle
		%			\hat{x}^{i,(k)}=\bar{x}^{\star,k}+\mathrm{e}^{-\lambda h}\left( x^{(k)}-\bar{x}^{\star,k}\right),\\
		%			\displaystyle
		%			x^{i,k+1}=\hat{x}^{i,(k)}-\sum^d_{l=1}\left( \hat{x}_l^{i,(k)}-\bar{x}^{\star,k}_l\right) \sigma \sqrt{h}Z_l^{(k)}e_l,
		%		\end{cases}\\
	%		&\mbox{Model C:}~ x^{i,k+1}=\bar{x}^{\star,k}+\sum^d_{l=1}\left( x_l^{i,k}-\bar{x}^{\star,k}_l\right)\left[ \exp\left( -(\lambda+\sigma^2/2)h+\sigma \sqrt{h}Z_l^{(k)}\right) \right]e_l,
	%		\end{split}
%\end{equation*}
%where the random variables $\{Z_l^{(k)}\}_{k,l}$ are i.i.d standard distributions. 
Most recently, Ha et al. \cite{HaConvergence2021} showed %this three models are special cases of 
a generalized discrete scheme of \eqref{cbo} and gave the convergence analysis without using the mean-field limit. The algorithm reads as follows:
\begin{equation*}
	\begin{cases}
		\displaystyle 
		x^{i,k+1}=x^{i,k}-\gamma\left( x^{i,k}-\bar{x}^{\star,k}\right) -\sum^d_{l=1}\left( x_l^{i,k}-\bar{x}^{\star,k}_l\right) \eta_l^k e_l,~i=1,\ldots,N,\\
		\displaystyle 
		\bar{x}^{\star,k}=\left( \bar{x}^{\star,k}_1,\ldots,\bar{x}^{\star,k}_d\right)^\mathrm{T} =\frac{\sum_{i=1}^N x^{i,k}\mathrm{e}^{-\beta f(x^{i,k})}}{\sum_{i=1}^N \mathrm{e}^{-\beta f(x^{i,k})}},
	\end{cases}
\end{equation*}
where  
%$$h:=\Delta t,\quad x^{(k)}:=x(kh),\quad k=0,1,\ldots.$$
the random variables $\{\eta_l^k\}_{k,l}$ are i.i.d with
\begin{equation*}
	\mathbb{E}[\eta_l^k]=0,\quad \mathbb{E}[|\eta_l^k|^2]=\zeta^2,\quad \zeta\geq0,\quad \forall k\geq0,\quad l=1,\ldots,d.
\end{equation*}
On the basis of \cite{HaConvergence2021}, Ha et al. \cite{HaTime2024} proposed a discrete momentum consensus-based optimization (Momentum-CBO) algorithm and provided a sufficient framework to guarantee the convergence of algorithm on the objective function values towards the global minimum. Nonetheless, in \cite{HaTime2024,HaConvergence2020,HaConvergence2021}, the convergence was proved only for the case that the objective function is smooth. Ha et al. \cite{HaConvergence2020} stated that it will be interesting to relax the regularity of the objective function to the less regular objective function, at least continuous one.

Inspired by the above works, this paper is interested in solving the following minimization problem 
\begin{equation}\label{problem}
	\min_{x\in\mathbb{R}^d} f(x),
\end{equation}
where $f:\mathbb{R}^d\rightarrow\mathbb{R}_+$ is nonsmooth, nonconvex and perhaps even non-Lipschitz continuous, and we assume the solution of \eqref{problem} is unique, denoted by $x^*$.
We employ the smoothing techniques since the theoretical analysis of the algorithm is challenging due to the non-smoothness of the objective function. Smoothing techniques which use a class of smoothing functions to solve the nonsmooth problems by the problem's structure are a series of efficient ways to overcome nonsmoothness in optimization. Chen \cite{ChenSmoothing2012} considered a class of smoothing methods for minimization problems where the feasible set is convex but the objective function is not convex and not differentiable. Bian and Chen \cite{BianA2020} proposed a smoothing proximal gradient algorithm for solving a class of constrained optimization problems with the objective function defined by the sum of a nonsmooth convex function and a cardinality function. Moreover, for a continuous function, there is a framework to construct its smooth approximation functions by convolution \cite{ChenSmoothing2012,HiriartConvex1993,RockafellarVariational1998}.

Inspired of the discrete CBO method and smoothing techniques, we propose a variant of discrete CBO algorithm, namely smoothing iterative CBO (SICBO) algorithm (see Algorithm \ref{alg-cbo}) and give its convergence analysis. More precisely, we firstly prove that there is a common consensus point $x_\infty$ such that iterative sequence generated by the SICBO algorithm exponentially converges to it in the suitable sense. And we discuss the optimality of the stochastic consensus point that $f(x_\infty)$ can be close to $f(x^*)$ as much as possible under some reasonable assumptions. Moreover, we also present an error estimate of the SICBO algorithm with respect to the global minimum. It is important to highlight that in the convergence analysis, compared to the results of \cite{HaConvergence2021}, the condition of the objective function $f$ is relaxed drastically from $f\in C^2_b$ to nonsmooth or even not Lipschitz continuous, where we call $f\in C^2_b$, if $f$ is a twice continuously differentiable function with bounded derivatives up to second-order. %Therefore, the theoretical result in this paper can be applied to a broader class of optimization problems.

We summarize the main contributions of this paper as follows.
\begin{itemize}
	\item[\textbullet] We present a variant of discrete CBO algorithm and provide the convergence analysis without resorting to the corresponding mean-field model. Taking advantage of smoothing methods, this paper achieves a breakthrough by proposing an effective iterative CBO algorithm for solving global minimizer of a nonconvex, nonsmooth and possible not Lipschitz continuous optimization problem with theoretical analysis. Moreover, this work also provides a good starting point for studying the convergence
	of the iterative CBO algorithm under weaker assumptions of the objective function.
	\item[\textbullet] We demonstrate that the solution process generated by the algorithm exponentially converges to a common point $x_\infty$ in expectation and almost surely converges to $x_\infty$ (see Theorem \ref{th-exp}). %These results are the important supplement to the convergence analysis of the iterative CBO algorithm, which is conducive to a better understanding of the convergence properties of the algorithm.
	\item[\textbullet] We provide a more detailed error estimation on the objective function values at the iterates of the  proposed algorithm towards its global minimum (see Corollary \ref{co-err}), which provides an important guidance for its application. To the best of our knowledge, there is no literature providing relevant theoretical proof on this topic yet. 
\end{itemize}

The remaining part of the paper proceeds as follows: Section \ref{sec2} is devoted to the introduction of some notations and preliminaries. In Section \ref{sec3}, we present the SICBO algorithm for solving problem \eqref{problem}. In Section \ref{sec4}, we study the global consensus of the proposed algorithm. The corresponding stochastic convergence results and the error estimates towards the global minimum are discussed in Section \ref{sec5}. In Section \ref{sec6}, we provide several numerical simulations to justify the efficiency of the proposed SICBO algorithm and the better performance of it comparing with other deterministic methods. Finally, Section \ref{sec7} gives the conclusions.

\section{Notations and preliminaries}\label{sec2}

\textbf{Notations}: The $d$-dimensional real-valued vector space is identified with $\mathbb{R}^d$ and $\mathbb{R}^d_+=[0,+\infty)^d$. $e_l$ is the column vector in $\mathbb{R}^d$, which takes the value $1$ at the $l$-th element and $0$ for the others. $\left\| \cdot\right\|$ denotes the $\ell_2$ norm of the corresponding vector or matrix, and $\left\| \cdot\right\|_1 $ denotes the $\ell_1$ norm. For the random variables $X$ and $Y$ in the probability space $(\Omega,\mathcal{B},\mathbb{P})$, $X\sim Y$ means both $X$ and $Y$ follow the same distribution. $X\sim \mathcal{N}(m,s^2)$ denotes that $X$ follows a normal distribution with mean $m$ and variance $s^2$. $\mathbb{E}[X]$ indicates the expectation of $X$. $\mathcal{P}_{ac}(\mathbb{R}^d)$ denotes the space of Borel probability measures that are absolutely continuous w.r.t the Lebesgue measure on $\mathbb{R}^d$.

In what follows, we present a brief summary of some relevant concepts and theorems, which are the key tools for designing algorithms and providing theoretical analysis. At the beginning, we recall the definitions of the stochastic consensus, including $L^p$ and almost surely (a.s.) global consensus.
\begin{definition}\label{def-con}
	Let $\{X_k\}_{k\geq 0}=\left\{x^{i,k}:1\leq i\leq N\right\}_{k\geq 0}$ be a discrete stochastic process.
	\begin{itemize}
		\item[{\rm(i)}] The stochastic process $\{X_k\}_{k\geq 0}$ exhibits  $L^p$ global consensus, $p\geq 1$, if
		$$\lim_{k\rightarrow\infty}\max_{i,j}\mathbb{E}\left[ \left\|x^{i,k}-x^{j,k} \right\|^p \right]=0.$$
		\item[{\rm(ii)}] The stochastic process $\{X_k\}_{k\geq 0}$ exhibits almost surely global consensus, if
		$$\lim_{k\rightarrow\infty} \left\|x^{i,k}-x^{j,k} \right\|=0, ~~\forall 1\leq i,j\leq N,~a.s..$$
	\end{itemize}
\end{definition}

The first result that we require is related to the martingale and its almost-sure convergence theorem. 
\begin{definition}[\cite{ResnickA1999}]\label{def-mar}
	Given integrable random variables $\{X_k\}_{k\geq 0}$ and $\sigma$-fields $\{\mathcal{B}_k\}_{k\geq 0}$, $\{(X_k,\mathcal{B}_k)\}_{k\geq 0}$ is a martingale if
	\begin{itemize}
		\item[{\rm(i)}] $\mathcal{B}_0\subset\mathcal{B}_1\subset\mathcal{B}_2\subset\ldots$;
		\item[{\rm(ii)}] $X_k$ is adapted in the sense that for each $k$, $X_k\in\mathcal{B}_k$, that is, $X_k$ is $\mathcal{B}_k$-measurable;
		\item[{\rm(iii)}] for $0\leq l<k$, $\mathbb{E}[X_k|\mathcal{B}_l]=X_l, ~\mbox{a.s.}$.
	\end{itemize}
\end{definition}
%\begin{lemma}\label{lem-martingale}
%	Suppose that $\{Z_k,k\geq 0\}$ is an independent sequence of integrable random variables satisfying for $k\geq 0$, $\mathbb{E}[Z_k]=0$. Set $X_0=0$, $X_k=\sum_{l=1}^kZ_l$, $k\geq1$, and $\mathcal{B}_k:=\sigma(Z_0,\ldots,Z_k)$. Then $\{(X_k,\mathcal{B}_k),k\geq 0\}$ is a martingale.
%	\end{lemma}
%\beginproof
%	By the definition of $X_k$, we have
%	\begin{equation*}
%		\mathbb{E}[X_{k+1}|\mathcal{B}_k]=\mathbb{E}\left[\sum_{l=1}^{k+1}Z_l\Bigg|\mathcal{B}_k\right]=\mathbb{E}[Z_{k+1}|\mathcal{B}_k]+\mathbb{E}\left[\sum_{l=1}^kZ_l\Bigg|\mathcal{B}_k\right]=\mathbb{E}[X_k|\mathcal{B}_k]\overset{\text{a.s.}}{=}X_k,
%	\end{equation*}
%where the third equality follows from the independence of $Z_{k+1}$ and $\mathcal{B}_k$, and the last equality follows from $X_k$ is $\mathcal{B}_k$-measurable. Therefore, $\{(X_k,\mathcal{B}_k),k\geq 0\}$ is a martingale.
%\endproof
\begin{theorem}[Doob's martingale convergence theorem \cite{ResnickA1999}]\label{th-martingale}
If martingale $\{(X_k,\mathcal{B}_k)\}_{k\geq 0}$ satisfies $\sup_k\mathbb{E}[|X_k|]<\infty$, then there exists $X_\infty\in L_1$ such that $X_k\rightarrow X_\infty$ a.s..
\end{theorem}

The second indispensable result is the concept of smoothing function, which is widely adopted in the fields of scientific computing and  optimization.
\begin{definition}[\cite{BianWorst2013}]\label{def-sf}
For the objective function $f$ in \eqref{problem},
we call $\tilde{f}:\mathbb{R}^d\times(0,\bar{\mu}]\rightarrow\mathbb{R}$ with $\bar{\mu}>0$ a smoothing function of it,
if the following conditions hold:
\begin{itemize}
	\item [{\rm (i)}] for any fixed $\mu\in(0,\bar{\mu}]$, $\tilde{f}(\cdot,\mu)$ is continuously differentiable in
	$\mathbb{R}^d$, and for any fixed $x\in\mathbb{R}^d$, $\tilde{f}(x,\cdot)$ is differentiable on $(0,\bar{\mu}]$;
	\item [{\rm (ii)}] for any fixed $x\in\mathbb{R}^d$, $\lim_{z\rightarrow x,\mu\downarrow0}\tilde{f}(z,\mu)=f(x)$;
	\item [{\rm (iii)}] for any bounded set $\mathcal{X}\subseteq\mathbb{R}^d$, there exist positive constants $\kappa$ and $\nu$, and $q\in[0,1)$ such that for any $x\in\mathcal{X}$ and $\mu\in(0,\bar{\mu}]$, it holds
	\begin{equation}\label{eq-mu1}
		\left|\nabla_{\mu}\tilde{f}(x,\mu)\right|\leq\kappa\mu^{-q} \quad \mbox{and} \quad \left\|\nabla^2_{xx}\tilde{f}(x,\mu)\right\|\leq \nu\mu^{-q-1}.
	\end{equation}
\end{itemize}
\end{definition}
Definition \ref{def-sf} (ii) and (iii) imply that
\begin{equation*}\label{eq-mu3}
\left|\tilde{f}(x,\mu)-f(x)\right|\leq\kappa\mu^{-q+1},~~\forall x\in\mathcal{X}, \mu\in(0,\bar{\mu}].
\end{equation*}
%\begin{remark}\label{re-sf}
%	According to the analysis in \cite{BianWorst2013}, the constant $q$ in Definition \ref{def-sf} can be classified depending on $f$ in the following two classes:
%	\begin{itemize}
%		\item [{\rm (i)}] if $f$ is locally Lipschitz continuous, then $q=0$;
%		\item [{\rm (ii)}] if $f=|g(x)|^p$ with differentiable function $g:\mathbb{R}^d\rightarrow \mathbb{R}$ and $0<p<1$, which is not Lipschitz continuous, then $q=1-p$.
%	\end{itemize}
%\end{remark}
\begin{remark}
In \cite{HiriartConvex1993,RockafellarVariational1998}, the authors introduced that the smoothing function of a continuous function $f:\mathbb{R}^d\rightarrow\mathbb{R}$ can be constructed by using convolution. 
Define function $\tilde{f}:\mathbb{R}^d\times\mathbb{R}_+\rightarrow\mathbb{R}$ by 
\begin{equation}\label{smoothing1}
	\tilde{f}(x,\mu):=\int_{\mathbb{R}^d}f(x-z)\psi^\mu(z)\mathrm{d}z=-\int_{\mathbb{R}^d}f(z)\psi^\mu(x-z)\mathrm{d}z,
\end{equation}
where $\psi^\mu$ is a sequence of bounded and nonnegative measurable functions with $\int_{\mathbb{R}^d}\psi^\mu(z)\mathrm{d}z=1$ such that the sets $B^\mu=\{z|\psi^\mu(z)>0\}$ form a bounded sequence converging to $\{0\}$ as $\mu$ converging to $0$. Rockafellar and Wets \cite[Example 7.19]{RockafellarVariational1998} gave a common used method of generating $\psi^\mu$. Considering a compact set $B$ of the origin and a continuous function $\psi:B\rightarrow[0,\infty]$ with $\int_{B}\psi(z)\mathrm{d}z=1$, they define
\begin{equation*}
	\psi^\mu(z):=
	\begin{cases}
		\displaystyle 
		\mu^{-d}\psi(z/\mu),~&z\in \mu B, \\
		\displaystyle 
		0,~&z\not\in \mu B.
	\end{cases}
\end{equation*}
Then, \eqref{smoothing1} can be expressed by
\begin{equation}\label{smoothing22}
	\tilde{f}(x,\mu)=-\int_{\mathbb{R}^d}f(z)\mu^{-d}\psi((x-z)/\mu)\mathrm{d}z=
	\int_{\mathbb{R}^d}f(x-\mu{z})\psi(z)\mathrm{d}z,
\end{equation}
which meets the prescription of Definition \ref{def-sf} under some  reasonable conditions.
\begin{itemize} 
	\item[\textbullet] From \eqref{smoothing1}, if $\psi^\mu$ is $r$-th continuously differentiable, then $\tilde{f}(\cdot,\mu)$ is $(r+1)$-th continuously differentiable. Hence, we can easily find that $\tilde{f}$ in \eqref{smoothing22} satisfies Definition \ref{def-sf}-(i).
	\item[\textbullet] In \cite[Theorem 9.67]{RockafellarVariational1998}, the authors showed that $\tilde{f}$ in \eqref{smoothing22} satisfies Definition \ref{def-sf}-(ii).
	\item[\textbullet] Based on \eqref{smoothing22}, when $f$ is locally Lipschtz continuous, we obtain the following estimations about $\kappa$ and $\nu$ in Definition \ref{def-sf}-(iii). On the one hand, if $\int_{\mathbb{R}^d}\|z\|\psi(z)\mathrm{d}z\leq\bar{\rho}<+\infty$, then
	$$\nabla_{\mu}\tilde{f}(x,\mu)=\int_{\mathbb{R}^d\setminus{D}_x}\nabla f(x-\mu{z})^{\mathrm{T}}z\psi(z)\mathrm{d}z,$$
	where $D_x$ is a set of measure $0$ such that $f(x-\mu z)$ is differentiable with respect to $z$ for any $z\not\in D_x$. 
	Thus, $|\nabla_{\mu}\tilde{f}(x,\mu)|\leq\kappa$, $\forall x\in\mathcal{X}$, $\mu\in(0,\bar{\mu}]$ with $\kappa=L\bar{\rho}$, where $L$ is a Lipschitz constant of $f$ over set $\mathcal{X}+\bar{\mu}B$.
	On the other hand, if $\psi$ is second order continuously differentiable and $\int_{\mathbb{R}^d}\|\nabla\psi(z)\|\mathrm{d}z\leq\bar{\nu}$, then
	$$\nabla^2_{xx}\tilde{f}(x,\mu)=-\int_{\mathbb{R}^d}f(z)\mu^{-d-2}\nabla^2\psi((x-z)/\mu)\mathrm{d}z.$$
	Thus, $\|\nabla^2_{xx}\tilde{f}(x,\mu)\|\leq\nu\mu^{-1}$ with $\nu=L\bar{\nu}$.
	To sum up, if $f$ is locally Lipschitz continuous on $\mathbb{R}^d$, we can always construct a smoothing function of it based on \eqref{smoothing22} with $q=0$. 
\end{itemize}
In what follows, we provide an example of constructing a smoothing function of plus function $t_+$ to accurately illustrate the statement above. Let $\rho: \mathbb{R}\rightarrow \mathbb{R}_+$ be a continuous density function satisfying
$\rho(s)=\rho(-s)$ and $\kappa:=\int_{-\infty}^\infty|s|\rho(s)\mathrm{d}s<\infty$.
Then 	
\begin{equation*}
	\phi(t,\mu):=\int_{-\infty}^\infty(t-\mu s)_+\rho(s)\mathrm{d}s
\end{equation*}
from $\mathbb{R}\times\mathbb{R}_+$ to $\mathbb{R}_+$ is well defined and satisfies $$\nabla_{\mu}\phi(t,\mu)=\int^{+\infty}_{t/\mu}s\rho(s)\mathrm{d}s, \,\nabla_t\phi(t,\mu)=\int_{-\infty}^{t/\mu}\rho(s)\mathrm{d}s,\, \nabla_t^2\phi(t,\mu)=\rho(t/\mu)/\mu.$$
Hence $\phi(t,\mu)$ is a smoothing function of $t_+$ defined in Definition \ref{def-sf} with $q=0$. Many nonsmooth functions in a large number of practical applications can be reformulated by using the plus function $t_+$ , such as $|t|$, $\max(t,y)$, $\min(t,y)$ and so on. On the basis of $\phi$, we can construct a class of smoothing functions for the nonsmooth functions which are composed of $t_+$ \cite{ChenSmoothing2012}. And in \cite[Lemma 1]{ChenSmoothing2012}, it has been shown that, for any $p\in(0,1]$, $\psi(t,\mu)^p=\left( \phi(t,\mu)+\phi(-t,\mu)\right)^p$ is a smoothing function of $|t|^p$ with $q=1-p$. Moreover, if $f:=|g(x)|^p$ with second order continuously differentiable function $g:\mathbb{R}^d\rightarrow \mathbb{R}$ and $p\in(0,1]$, then $\tilde{f}(x,\mu)=\psi(g(x),\mu)^p$ is a smoothing function of $f$ with $q=1-p$. We also notice that the smoothing function defined in Definition \ref{def-sf} has the additivity, that is if $\tilde{f}_1$ and $\tilde{f}_2$ are smoothing functions of $f_1$ and $f_2$, respectively, then $\alpha_1\tilde{f}_1+\alpha_2\tilde{f}_2$ is a smoothing function of $\alpha_1f_1+\alpha_2f_2$ with any $\alpha_1\geq 0$ and $\alpha_2\geq0$.  The readers are referred to \cite{BianNeural2014,ChenSmoothing2012} for more details on this topic. 
\end{remark}
\begin{example}
Consider the following funciton
$$f(x)=\sum_{i=1}^m|a_i^\mathrm{T}x-b_i|^p$$
with $0<p\leq 1$ and $a_i\in\mathbb{R}^d$, $b_i\in\mathbb{R}, i=1,\ldots,m$.
Choose
\begin{equation}\label{smoothing}
	\varphi_1(s,\mu)=
	\begin{cases}
		\displaystyle |s|, &\mbox{if}~|s|>\mu\\
		\displaystyle \frac{s^2}{2\mu}+\frac{\mu}{2}, &\mbox{if}~|s|\leq\mu\\
	\end{cases}
\end{equation}
as a smoothing function of $|s|$. Then
$$\tilde{f}(x,\mu)=\sum_{i=1}^m\varphi_1(a_i^\mathrm{T}x-b_i,\mu)^p$$
is a smoothing function of $f$ with $\kappa=2^{-p}mp$, $\nu=2^{1-p}p\left\| \sum_{i=1}^ma_ia_i^\mathrm{T}\right\|+2^{2-p}p(1-p)\sum_{i=1}^m\left\| a_i\right\|^2  $ and $q=1-p$.
\end{example}

The last key result that we need is given as follows, which is obtained from Laplace principle \cite{DemboLarge1998}.
\begin{definition}[\cite{ResnickA1999}]
Let $\left(\Omega,\mathcal{B} \right)$ be measurable space. Let $\rho_1$ and $\rho_2$ be positive bounded measures on $\left(\Omega,\mathcal{B} \right)$. We say that $\rho_1$ is absolutely continuous with respect to $\rho_2$, written $\rho_1\ll\rho_2$, if for any measurable set $A$, $\rho_2(A)=0$ implies $\rho_1(A)=0$.
\end{definition}
\begin{proposition}[\cite{PinnauA2016}]\label{Laplace}
Assume that $f:\mathbb{R}^d\rightarrow\mathbb{R}_+$ is continuous, bounded and attains the global minimum at the unique point $x^*$, then for any probability measure $\rho\in\mathcal{P}_{ac}(\mathbb{R}^d)$, it holds
\begin{equation*}\label{laplace}
	\lim\limits_{\beta\rightarrow\infty}\left(-\frac{1}{\beta}\log\left(\int_{\mathbb{R}^d}\mathrm{e}^{-\beta f(x)}\mathrm{d}\rho(x)\right)\right)=f(x^*).
\end{equation*}
\end{proposition}
%\begin{remark}
%	By \eqref{laplace}, it is obvious to know
%	$$\lim_{\beta\rightarrow\infty}-\frac{1}{\beta}\log\mathbb{E}\left[ \mathrm{e}^{-\beta f(x)}\right] =f(x^*).$$
%\end{remark}
\section{Smoothing iterative CBO algorithm}\label{sec3}
Taking advantages of the smoothing method, we put forward the iterative CBO algorithm with smoothing function, namely Smoothing Iterative CBO (SICBO) algorithm, for solving the global minimizer of problem \eqref{problem}. %We pick a positive sequence $\{\mu_k\}$.
%We pick a non-increasing sequence of positive reals $\{\mu_k\}$ with $\lim_{k\rightarrow\infty}\mu_k=0$ and $0<\mu_0\leq\bar{\mu}$.
Details of the proposed algorithm are displayed in Algorithm \ref{alg-cbo}.
\begin{algorithm}
\renewcommand{\algorithmicrequire}{\textbf{Input:}}
\renewcommand{\algorithmicensure}{\textbf{Output:}}
\caption{Smoothing Iterative CBO (SICBO) algorithm}
\label{alg-cbo}
\begin{algorithmic}[0]
	\REQUIRE
	the number of particles $N$, initial points $x^{i,0}\in\mathbb{R}^d$, $i=1,\ldots,N$,
	%, and $\mu_0\in(0,\bar{\mu}]$.
	parameters $\beta>0$, $\gamma> 0$, $\zeta\geq0$, and a positive sequence $\{\mu_k\}$ converging to $0$.\\
	Set $k=0$.
	\WHILE{a termination criterion is not met,}
	\STATE \textbf{Step 1.} Compute the weighted average point:
	\STATE 
	\begin{equation}\label{eq-weightedstate}
		\bar{x}^{\star,k}=\left( \bar{x}^{\star,k}_1,\ldots,\bar{x}^{\star,k}_d\right)^\mathrm{T} =\frac{\sum_{i=1}^N x^{i,k}\mathrm{e}^{-\beta \tilde{f}(x^{i,k},\mu_k)}}{\sum_{i=1}^N \mathrm{e}^{-\beta \tilde{f}(x^{i,k},\mu_k)}}.
	\end{equation}
	%based on Equation 
	\STATE \textbf{Step 2.} Generate the i.i.d random variables $\{\eta_l^k\}_{k,l}$ that satisfy
	\STATE
	\begin{equation}\label{eta}
		%\mathbb{E}[\eta_l^k]=0,\quad \mathbb{E}[|\eta_l^k|^2]=\zeta^2,\quad l=1,\ldots,d.
		\eta_l^k\sim\mathcal{N}(0,\zeta^2),\quad l=1,\ldots,d.
	\end{equation}
	\STATE \textbf{Step 3.} Update $x^{i,k+1}$ for $i=1,\ldots,N$ by
	\STATE 
	\begin{equation} \label{eq-updatestate}
		x^{i,k+1}=x^{i,k}-\gamma\left( x^{i,k}-\bar{x}^{\star,k}\right) -\sum^d_{l=1}\left( x_l^{i,k}-\bar{x}^{\star,k}_l\right) \eta_l^k e_l.
	\end{equation}
	%		\STATE Update smoothing parameter $\mu_{k+1}$:
	%		\begin{equation}\label{eq-updatemu}
		%			\mu_{k+1}=\mathrm{e}^{b}\mu_k.
		%		\end{equation}
	\STATE \textbf{Step 4.} Increment $k$ by one and return to \textbf{Step 1}.
	\ENDWHILE
	\ENSURE
	$x^{i,k+1}$, $i=1,\ldots,N$.
\end{algorithmic}
\end{algorithm}
\begin{remark}\label{re-stop}
For the SICBO algorithm, there are some possible stopping criterions that we can choose as follows: 
\begin{itemize}
	\item [{\rm(i)}] $k\leq K$, where $K$ represents the given maximum number of iterations;
	\item [{\rm(ii)}] $\max_{1\leq i\leq N} \left\| x^{i,k+1}-x^{i,k}\right\| \leq \varepsilon$ with $\varepsilon>0$;
	\item [{\rm(ii)}] $\max_{1\leq i\leq N} \left\| x^{i,k+1}-x^{i,k}\right\| \leq \varepsilon_1 $  and   $\max_{1\leq i\leq N} \frac{\left| f(x^{i,k+1})-f(x^{i,k})\right| }{\left\| x^{i,k+1}-x^{i,k}\right\|} \leq \varepsilon_2 $ with $\varepsilon_1,\varepsilon_2>0$;
	\item [{\rm(iv)}] $\max_{1\leq i\leq N} \left| f(x^{i,k+1})-f(x^{i,k})\right|\leq\varepsilon$ with $\varepsilon>0$.
\end{itemize}
\end{remark}
%\begin{remark}
%	In practice, the i.i.d random variables $\eta_l^k$ in Algorithm \ref{alg-cbo} usually sampled from the normal distribution, i.e. $\eta_l^k\sim\mathcal{N}(0,\zeta^2)$.
%\end{remark}

\section{Global consensus}\label{sec4}
The proposed SICBO algorithm can be understood as a multi-agent system \cite{HaTime2024}. Global consensus means that agents achieve an agreement on their common state by using information from communication mechanisms. From the view point of opinion dynamics, the emergence of global consensus also shows that all agents share a common opinion on the proposal.
In this section, based on Definition \ref{def-con}, we aim at providing the analysis with regard to the global consensus for the SICBO algorithm under some appropriate conditions on parameters $\gamma$ and $\zeta$.

\begin{theorem}\label{th-consensus}
Suppose $\{x^{i,k}\}_{i=1}^N$ is the solution process of the SICBO algorithm, then the following results hold for $\forall i,j=1,\ldots,N$.
\begin{itemize}
	\item [{\rm(i)}] Assume that the parameter $\gamma$ in the SICBO algorithm satisfies
	\begin{equation}\label{par2}
		\left| 1-\gamma\right| <1,
	\end{equation}
	then $$\mathbb{E}\left[ \left\| x^{i,k}-x^{j,k}\right\| \right] =\mathcal{O}\left(|1-\gamma|^k \right), \mbox{as}~ k\rightarrow\infty.$$
	%which means that $\mathbb{E}\left[ \left\| x^{i,k}-x^{j,k}\right\| \right]$ asymptotically converges to $0$ in an exponential rate.
	\item [{\rm(ii)}] Assume that the parameters $\gamma$ and $\zeta$ in the SICBO algorithm satisfy
	\begin{equation}\label{par1}
		(1-\gamma)^2+\zeta^2<1,
	\end{equation}
	then $$\mathbb{E}\left[\left\| x^{i,k}-x^{j,k}\right\|^2\right]=\mathcal{O}\left( \left( (1-\gamma)^2+\zeta^2\right)^k \right), \mbox{as}~ k\rightarrow\infty.$$
	%which means that SICBO algorithm reaches $L^2$ global consensus asymptotically in an exponential rate.
	Moreover, there exists a positive constant $\alpha$ and a positive random variable $C_1(\omega)$ such that
	\begin{equation*}
		\left| x^{i,k}_l-x^{j,k}_l\right|^2\leq C_1(\omega)\mathrm{e}^{-\alpha k},~\mbox{a.s.},~ \forall k\geq 0,~\forall l=1,\ldots,d.
	\end{equation*}
	%which means that SICBO algorithm reaches almost surely global consensus asymptotically in an exponential rate.
\end{itemize}
\end{theorem}
\begin{proof}
(i) Following from \eqref{eq-updatestate}, one has that for $i,j=1,\ldots,N$ and $l=1,\ldots,d$,
\begin{equation*}
x^{i,k+1}_l-x^{j,k+1}_l=\left( 1-\gamma-\eta_l^k\right)\left( x^{i,k}_l-x^{j,k}_l\right).  
\end{equation*}
By the above recursive relation, we get
\begin{equation}\label{eq-xk0}
x^{i,k}_l-x^{j,k}_l=\left( x^{i,0}_l-x^{j,0}_l\right)\prod_{n=0}^{k-1}\left( 1-\gamma-\eta_l^n\right). 
\end{equation}
Taking expectation on the both sides of \eqref{eq-xk0}, we have 
\begin{equation}\label{i-j}
\begin{split}
	\mathbb{E}\left[ x^{i,k}_l-x^{j,k}_l\right]&=\mathbb{E}\left[ x^{i,0}_l-x^{j,0}_l\right]\mathbb{E}\left[ \prod_{n=0}^{k-1}\left( 1-\gamma-\eta_l^n\right)\right]\\
	&=\mathbb{E}\left[ x^{i,0}_l-x^{j,0}_l\right]\prod_{n=0}^{k-1}\mathbb{E}\left[ 1-\gamma-\eta_l^n\right]\\
	&=(1-\gamma)^k\mathbb{E}\left[ x^{i,0}_l-x^{j,0}_l\right],
\end{split}
\end{equation}
where the first equality follows from the independence of $\left\lbrace \eta_l^n\right\rbrace_{n\geq0} $ and $x^{i,0}_l-x^{j,0}_l$, and the last two equalities follow from \eqref{eta}.
Considering $\left| 1-\gamma\right| <1$, then we have
$$\mathbb{E}\left[\left\|  x^{i,k}-x^{j,k}\right\| \right] =\mathcal{O}\left(|1-\gamma|^k \right),~\mbox{as} ~k\rightarrow\infty.$$
(ii) By squaring \eqref{eq-xk0}, we obtain
\begin{equation}\label{eq-xk02}
\left| x^{i,k}_l-x^{j,k}_l\right|^2=\left| x^{i,0}_l-x^{j,0}_l\right|^2\prod_{n=0}^{k-1}\left| 1-\gamma-\eta_l^n\right|^2. 
\end{equation}
One takes expectation on the both sides of \eqref{eq-xk02} to get
\begin{equation*}
\begin{split}
	\mathbb{E}\left[ \left| x^{i,k}_l-x^{j,k}_l\right|^2\right] &=\mathbb{E}\left[ \left| x^{i,0}_l-x^{j,0}_l\right|^2\right] \prod_{n=0}^{k-1}\mathbb{E}\left[ \left| 1-\gamma-\eta_l^n\right|^2\right] \\
	&=\prod_{n=0}^{k-1}\mathbb{E}\left[ (1-\gamma)^2-2(1-\gamma)\eta_l^n+(\eta_l^n)^2\right]  \mathbb{E}\left[ \left|x^{i,0}_l-x^{j,0}_l\right|^2\right]\\
	&=\left((1-\gamma)^2+\zeta^2\right)^k\mathbb{E}\left[ \left| x^{i,0}_l-x^{j,0}_l\right|^2\right] ,
\end{split}
\end{equation*}
where the first equality can be proved by the independence of $\left\lbrace \eta_l^n\right\rbrace_{n\geq0} $ and $x^{i,0}_l-x^{j,0}_l$, and the last one follows from \eqref{eta}.
By the above result, we conclude that
$$\mathbb{E}\left[ \left\| x^{i,k}-x^{j,k}\right\| ^2\right] =\mathcal{O}\left( \left( (1-\gamma)^2+\zeta^2\right)^k \right),~\mbox{as} ~k\rightarrow\infty.$$

In addition, noting \eqref{eq-xk02} and
$$\left( 1-\gamma-\eta_l^n\right)^2\leq\mathrm{e}^{\left( 1-\gamma-\eta_l^n\right)^2-1}%=\mathrm{e}^{-(\gamma+\eta_l^n)(2-\gamma-\eta_l^n)}
,$$
we obtain
\begin{equation*}
\begin{split}
	\left| x^{i,k}_l-x^{j,k}_l\right|^2&\leq\left| x^{i,0}_l-x^{j,0}_l\right|^2\prod_{n=0}^{k-1}\mathrm{e}^{\left( 1-\gamma-\eta_l^n\right)^2-1}\\
	&=\left| x^{i,0}_l-x^{j,0}_l\right|^2\exp\left( -\sum_{n=0}^{k-1}\left(1- \left( 1-\gamma-\eta_l^n\right)^2\right) \right).
\end{split}
\end{equation*}
%To ease notation, let $A_k:=\frac{1}{k}\sum_{n=0}^{k-1}(\gamma+\eta_l^n)(2-\gamma-\eta_l^n)$. 
Then following from the Strong Law of Large Numbers \cite{ResnickA1999} and \eqref{eta}, we have
\begin{equation}\label{slln}
	\begin{split}
\lim_{k\rightarrow\infty}\frac{1}{k}\sum_{n=0}^{k-1}\left(1-\left(  1-\gamma-\eta_l^n\right)^2\right)
 &=\mathbb{E}\left[1-\left( 1-\gamma-\eta_l^n\right)^2\right]\\
 &=1-(1-\gamma)^2-\zeta^2>0,~\mbox{a.s.}.
\end{split}
\end{equation}
Choose a constant $\alpha$ such that
\begin{equation*}
0<\alpha<1-(1-\gamma)^2-\zeta^2,
\end{equation*}
then there exists an integer $K>0$ such that 
\begin{equation*}\label{eq-K}
\left| x^{i,k}_l-x^{j,k}_l\right|^2\leq\left| x^{i,0}_l-x^{j,0}_l\right|^2\mathrm{e}^{-\alpha k},\quad \forall k\geq K,~\mbox{a.s.}.
\end{equation*}
Hence, there exists a random variable $C_1(\omega):=\max_{0\leq k\leq K}\left\{\mathrm{e}^{\alpha k}\left|x^{i,k}_l-x^{j,k}_l\right|^2\right\}$ satisfying
\begin{equation*}
\left| x^{i,k}_l-x^{j,k}_l\right|^2\leq C_1(\omega)\mathrm{e}^{-\alpha k},\quad \forall k\geq 0,~\mbox{a.s.}.
\end{equation*}
\end{proof}
\begin{remark}
From Theorem \ref{th-consensus}, we know that
\begin{itemize}
	\item [{\rm(i)}] if \eqref{par2} holds, $\mathbb{E}\left[ \left\| x^{i,k}-x^{j,k}\right\| \right]$ asymptotically converges to $0$ in an exponential rate;
	\item [{\rm(ii)}] if \eqref{par1} holds, then the SICBO algorithm reaches $L^2$ global consensus and almost surely global consensus asymptotically in an exponential rate.
\end{itemize}
\end{remark}
\section{Convergence analysis and error estimation}\label{sec5}
In this section, we give rigorous convergence properties and error estimates of the stochastic process $\{x^{i,k}\}_{i=1}^N$ generated by the SICBO algorithm.
\subsection{Convergence analysis}
In this subsection, we show that $x^{i,k}$ tends to a common asymptotic point $x_\infty$ as $k\rightarrow\infty$ a.s. for all $i=1,\ldots,N$. Firstly, we denote $\bar{x}^{k}$ as the average point at iterate $k$ with the following definition:
\begin{equation*}
	\bar{x}^{k}=\left( \bar{x}^{k}_1,\ldots,\bar{x}^{k}_d\right)^\mathrm{T}:=\frac{1}{N}\sum_{i=1}^{N}x^{i,k}.
\end{equation*}
The first result gives an estimation between $x^{i,k}$ and $\bar{x}^{\star,k}$.
\begin{lemma}\label{lem-e}
Suppose that $\{x^{i,k}\}_{i=1}^N$ be a solution process of the SICBO algorithm, define $\mathcal{D}_l:=\max_{1\leq i\leq N}\left| x^{i,0}_l-\bar{x}^{0}_l\right|^2$, then the following inequality holds:
\begin{equation}\label{eq-existar}
	\frac{1}{N}\sum_{i=1}^N\mathbb{E}\left[ \left\| x^{i,k}-\bar{x}^{\star,k}\right\|^2\right] \leq2\left( (1-\gamma)^2+\zeta^2\right)^k\sum_{l=1}^d \mathbb{E}\left[ \mathcal{D}_l \right].
\end{equation}
\end{lemma}
\begin{proof}
Following from \eqref{eq-weightedstate}, we first note that
\begin{equation}\label{eq-max}
\left\| \bar{x}^{k}-\bar{x}^{\star,k}\right\|^2\leq\max_{1\leq i\leq N}\left\| \bar{x}^{k}-x^{i,k}\right\| ^2.
\end{equation}
Taking average of formula \eqref{eq-updatestate} for $i=1,\ldots,N$, we have
\begin{equation}\label{eq-barx}
\bar{x}^{k+1}-\bar{x}^{k}=-\gamma\left( \bar{x}^{k}-\bar{x}^{\star,k}\right)-\sum_{l=1}^d\left( \bar{x}^{k}_l-\bar{x}^{\star,k}_l\right)\eta_l^ke_l.  
\end{equation}
Subtracting \eqref{eq-barx} from \eqref{eq-updatestate}, we have
$$x^{i,k+1}_l-\bar{x}^{k+1}_l=\left(1-\gamma-\eta_l^k\right)\left( x^{i,k}_l-\bar{x}^{k}_l\right),\quad \forall l=1,\ldots,d.$$
Then using the above relation recursively, one gets
\begin{equation*}
x^{i,k}_l-\bar{x}^{k}_l=\left( x^{i,0}_l-\bar{x}^{0}_l\right) \prod_{n=0}^{k-1}\left( 1-\gamma-\eta_l^n\right), \quad \forall l=1,\ldots,d,
\end{equation*}
which implies that
\begin{equation}\label{eq-xibar}
\left\| x^{i,k}-\bar{x}^{k}\right\|^2=\sum_{l=1}^d\left( \left| x^{i,0}_l-\bar{x}^{0}_l\right|^2 \prod_{n=0}^{k-1}\left( 1-\gamma-\eta_l^n\right)^2\right) .
\end{equation}
Then one can obtain that
\begin{equation}\label{eq-xistar}
\begin{split}
	\frac{1}{N}\sum_{i=1}^N\left\| x^{i,k}-\bar{x}^{\star,k}\right\|^2&=\frac{1}{N}\sum_{i=1}^N\left\| x^{i,k}-\bar{x}^{k}+\bar{x}^{k}-\bar{x}^{\star,k}\right\|^2 \\
	%&=\frac{1}{N}\sum_{i=1}^N\left( \left\| x^{i,k}-\bar{x}^{k}\right\|^2+2\left( x^{i,k}-\bar{x}^{k}\right) \left( \bar{x}^{k}-\bar{x}^{\star,k}\right) +\left\| \bar{x}^{k}-\bar{x}^{\star,k}\right\|^2 \right)\\
	&=\frac{1}{N}\sum_{i=1}^N\left( \left\| x^{i,k}-\bar{x}^{k}\right\|^2+\left\| \bar{x}^{k}-\bar{x}^{\star,k}\right\|^2 \right)\\
	&\leq2\max_{1\leq i\leq N}\left\| \bar{x}^{k}-x^{i,k}\right\| ^2\\
	&\leq2\sum_{l=1}^d\left( \mathcal{D}_l \prod_{n=0}^{k-1}\left( 1-\gamma-\eta_l^n\right)^2\right),
\end{split}
\end{equation}
where the second equality follows from the definition of $\bar{x}^{k}$, the first inequality follows from \eqref{eq-max}, and the last equality follows from \eqref{eq-xibar}.
Finally, using \eqref{eta} and taking expectation on the both sides of \eqref{eq-xistar}, we get
\begin{equation*}
\begin{split}
	\frac{1}{N}\sum_{i=1}^N\mathbb{E}\left[ \left\| x^{i,k}-\bar{x}^{\star,k}\right\|^2\right] &\leq 2\sum_{l=1}^d\mathbb{E}\left[ \mathcal{D}_l\right] \prod_{n=0}^{k-1}\mathbb{E}\left[\left( 1-\gamma-\eta_l^n\right)^2\right]\\
	&\leq 2\left( (1-\gamma)^2+\zeta^2\right)^k\sum_{l=1}^d \mathbb{E}\left[ \mathcal{D}_l \right].
\end{split}
\end{equation*}
\end{proof}
\begin{remark}
Lemma \ref{lem-e} along with the fact that $1+t\leq\mathrm{e}^t,~\forall t\in\mathbb{R}$ leads to
\begin{equation}\label{eq-ex}
	\frac{1}{N}\sum_{i=1}^N\mathbb{E}\left[ \left\| x^{i,k}-\bar{x}^{\star,k}\right\|^2\right]\leq 2\mathrm{e}^{-(2\gamma- \gamma^2-\zeta^2)k}\sum_{l=1}^d \mathbb{E}\left[ \mathcal{D}_l \right] .
\end{equation}
\end{remark}

The following lemma plays a vital role in analyzing the convergence of the proposed method.
\begin{lemma}\label{lem-martingale}
For any given $l\in\{1,\ldots,d\}$, denote $\mathcal{F}_k=\sigma\left(\eta_l^k,\ldots,\eta_l^0\right)$, then $\left\{\left( \sum_{n=0}^{k}\left( x^{i,n}_l-\bar{x}^{\star,n}_l\right)\eta_l^n,\mathcal{F}_k\right)\right\}_{k\geq0}$ is a martingale.
\end{lemma}
\begin{proof}
By the definition of $\mathcal{F}_k$ and \eqref{eq-updatestate}, we have $\mathcal{F}_0\subset\mathcal{F}_1\subset\ldots$, and
\begin{equation}\label{measure}
\sum_{n=0}^{k}\left( x^{i,n}_l-\bar{x}^{\star,n}_l\right)\eta_l^n\in\mathcal{F}_k.
\end{equation} 
Besides, for any $k\geq0$, it holds
\begin{equation*}
\begin{split}
	&\mathbb{E}\left[ \sum_{n=0}^{k+1}\left( x^{i,n}_l-\bar{x}^{\star,n}_l\right)\eta_l^n\Bigg|\mathcal{F}_{k}\right]\\
	=&\mathbb{E}\left[ \sum_{n=0}^{k}\left( x^{i,n}_l-\bar{x}^{\star,n}_l\right)\eta_l^n\Bigg|\mathcal{F}_{k}\right]+\mathbb{E}\left[ \left( x^{i,k+1}_l-\bar{x}^{\star,k+1}_l\right)\eta_l^{k+1}\Big|\mathcal{F}_{k}\right]\\
	=&\sum_{n=0}^{k}\left( x^{i,n}_l-\bar{x}^{\star,n}_l\right)\eta_l^n+\left( x^{i,k+1}_l-\bar{x}^{\star,k+1}_l\right)\mathbb{E}\left[\eta_l^{k+1} \Big|\mathcal{F}_{k}\right]\\
	=& \sum_{n=0}^{k}\left( x^{i,n}_l-\bar{x}^{\star,n}_l\right)\eta_l^n,
\end{split}
\end{equation*}
where the second equality follows from \eqref{measure} and the fact that $x^{i,k+1}_l-\bar{x}^{\star,k+1}_l\in\mathcal{F}_k$, and the last equality follows from \eqref{eta}. Therefore, we know from Definition \ref{def-mar} that $\left\{\left( \sum_{n=0}^{k}\left( x^{i,n}_l-\bar{x}^{\star,n}_l\right)\eta_l^n,\mathcal{F}_k\right)\right\}_{k\geq0}$ is a martingale.
\end{proof}

Based on the aforementioned lemmas, the following theorem reveals the convergence of the SICBO algorithm. %And in what follows, for notational simplicity, we denote 
%$$\mathcal{D}_l:=\max_{1\leq i\leq N}\left| x^{i,0}_l-\bar{x}^{0}_l\right|^2.$$
\begin{theorem}\label{th-common}
Suppose $\{x^{i,k}\}_{i=1}^N$ be the solution process of the SICBO algorithm and the parameters satisfy \eqref{par1}. Then this solution process converges to a common consensus point $x_\infty\in\mathbb{R}^d$ almost surely, i.e.
$$\lim_{k\rightarrow\infty}x^{i,k}=x_\infty,~\forall 1\leq i\leq N,~a.s..$$ 
\end{theorem}
\begin{proof}
It follows from \eqref{eq-updatestate} that
\begin{equation}\label{eq-xi}
x^{i,k}_l=x^{i,0}_l-\gamma\sum_{n=0}^{k-1}\left( x^{i,n}_l-\bar{x}^{\star,n}_l\right)-\sum_{n=0}^{k-1}\left( x^{i,n}_l-\bar{x}^{\star,n}_l\right)\eta_l^n. 
\end{equation}

On the one hand, we first establish the a.s. convergence of  $\sum_{n=0}^{k-1}\left( x^{i,n}_l-\bar{x}^{\star,n}_l\right)$ in what follows. By \eqref{eq-xistar}, one has 
\begin{equation*}
\begin{split}
	\left| x^{i,n}_l-\bar{x}^{\star,n}_l\right|&\leq\sqrt{\sum_{i=1}^N\left\| x^{i,n}-\bar{x}^{\star,n}\right\|^2 }\\
	&\leq\sqrt{2N\sum_{l=1}^d\left(\mathcal{D}_l  \prod_{m=0}^{n-1}\left( 1-\gamma-\eta_l^m\right)^2\right)}\\
	&\leq\sqrt{2N\sum_{l=1}^d\left(\mathcal{D}_l \prod_{m=0}^{n-1}\exp\left(( 1-\gamma-\eta_l^m)^2-1\right)\right) }\\
	&=\sqrt{2N\sum_{l=1}^d\left(\mathcal{D}_l \exp\left( \sum_{m=0}^{n-1}\left(( 1-\gamma-\eta_l^m)^2-1\right)\right)\right) },~\forall i=1,\ldots,N,
\end{split}
\end{equation*}
where the last inequality uses the relation $1+t\leq\mathrm{e}^t,~\forall t\in\mathbb{R}$, and $\mathcal{D}_l$ is defined as in Lemma \ref{lem-e}.
Then, \eqref{slln} implies that
\begin{equation*}
\lim_{n\rightarrow\infty}\frac{1}{n}\sum_{m=0}^{n-1}\left((1-\gamma-\eta_l^m)^2-1\right)
=\gamma^2-2\gamma+\zeta^2<0,~\mbox{a.s.}.
\end{equation*}
The above results deduce that there exist positive constants $c_1:=c_1(\omega)$ and $c_2:=c_2(\omega)$, depending only on $\omega\in\Omega$, such that
\begin{equation}\label{eq-bound}
\left| x^{i,n}_l-\bar{x}^{\star,n}_l\right|\leq c_1\mathrm{e}^{-c_2n},~\forall l=1,\ldots,d,~\mbox{a.s.}.
\end{equation}
In order to prove the convergence of $\sum_{n=0}^{k-1}\left( x^{i,n}_l-\bar{x}^{\star,n}_l\right)$, we only need to prove the convergence of $A_k:=\sum_{n=0}^{k-1}\left( x^{i,n}_l-\bar{x}^{\star,n}_l-c_1\mathrm{e}^{-c_2n}\right)$. By \eqref{eq-bound}, we know that $A_k$ is non-increasing on $k$ and
$$A_k\geq-2\sum_{n=0}^{k-1}c_1\mathrm{e}^{-c_2n}\geq-\frac{2c_1}{1-\mathrm{e}^{-c_2}},$$
which means that $A_k$ is bounded from below.
As a result, $A_k$ is almost surely convergent as $k$ tends to $\infty$, so is $\sum_{n=0}^{k-1}\left( x^{i,n}_l-\bar{x}^{\star,n}_l\right)$.

On the other hand, we show the a.s. convergence of $\sum_{n=0}^{k-1}\left( x^{i,n}_l-\bar{x}^{\star,n}_l\right)\eta_l^n$. By Lemma \ref{lem-martingale}, we know that
$\left\lbrace \left( \sum_{n=0}^{k-1}\left( x^{i,n}_l-\bar{x}^{\star,n}_l\right)\eta_l^n,\mathcal{F}_{k-1}\right) \right\rbrace _{k\geq0}$ is a martingale. Then, one has 
\begin{equation*}
\begin{split}
	\mathbb{E}\left[ \sum_{n=0}^{k-1}\left( x^{i,n}_l-\bar{x}^{\star,n}_l\right)\eta_l^n\right] ^2%&=\mathbb{E}\left[ \sum_{n=0}^{k-1}\left( x^{i,n}_l-\bar{x}^{\star,n}_l\right)^2(\eta_l^n)^2\right]\\
	&=\zeta^2\sum_{n=0}^{k-1}\mathbb{E}\left[ \left( x^{i,n}_l-\bar{x}^{\star,n}_l\right)^2\right] \\
	&\leq2N\zeta^2\left( \sum_{n=0}^{k-1}\left( (1-\gamma)^2+\zeta^2\right)^n\right) \sum_{l=1}^d \mathbb{E}\left[ \mathcal{D}_l  \right] \\
	&\leq\frac{2N\zeta^2}{2\gamma-\gamma^2-\zeta^2}\sum_{l=1}^d \mathbb{E}\left[ \mathcal{D}_l  \right] ,
\end{split}
\end{equation*}
where the first equality follows from \eqref{eta} and the independence of $\eta_l^n$ and $x^{i,n}_l-x^{j,n}_l$ for any given $ n\geq 0$ and $l=1,\ldots,d$, the first inequality follows from \eqref{eq-existar}, and the last inequality is deduced by \eqref{par1}. %The preceding inequality yields $\sum_{n=0}^{k-1}\left( x^{i,n}_l-\bar{x}^{\star,n}_l\right)\eta_l^n$ is $L^2$-bounded. 
Thus, making use of Theorem \ref{th-martingale}, the limit of $\sum_{n=0}^{k-1}\left( x^{i,n}_l-\bar{x}^{\star,n}_l\right)\eta_l^n$ exists a.s..

By \eqref{eq-xi}, we have shown that $\forall i=1,\ldots,N$, there exists a random variable $x^i_\infty$ such that
$$\lim_{k\rightarrow\infty}x^{i,k}=x^i_\infty,~\forall l=1,\ldots,d,~\mbox{a.s.}.$$
Consequently, following from Theorem \ref{th-consensus} (iii), we establish the desired results.
\end{proof}

In the following, we derive some theoretical results on the convergence rate of the iterative process of the SICBO algorithm. We first give a supporting lemma, whose detailed proof can be found in \cite{KoConvergence2022}.
\begin{lemma}[\cite{KoConvergence2022}]\label{lem-Y}
Let $\{Y_k\}_{k\geq0}$ be an i.i.d. sequence of random variables with $\mathbb{E}\left[ |Y_0|\right] <\infty$. Then, for any given $\varepsilon>0$, there exists a random variable $C_2(\omega)$ satisfying
\begin{equation*}
	|Y_k|<C_2(\omega)\mathrm{e}^{\varepsilon k},\quad \forall k\geq0,~~a.s..
\end{equation*}
\end{lemma}

Next, the exponentially convergence of the iterative process by the SICBO algorithm is shown in the next theorem.
\begin{theorem}\label{th-exp}
Suppose $\{x^{i,k}\}_{i=1}^N$ be the solution process of the SICBO algorithm, then for $\forall i=1,2,\ldots,N$, the following statements hold.
\begin{itemize}
	\item [{\rm (i)}] If the parameter $\gamma$ in the SICBO algorithm satisfies \eqref{par2}, then
	\begin{equation}\label{econvergence}
		\mathbb{E}\left[\left\| x_\infty-x^{i,k}\right\|_1\right]\leq\frac{(\gamma+\zeta)(1-\gamma)^k}{\gamma}\mathbb{E}\left[\max_{1\leq j\leq N}\left\|x^{i,0}-x^{j,0} \right\|_1 \right], \quad \forall k\geq0,
	\end{equation}
	where $x_\infty$ is defined as in Theorem \ref{th-common}. Therefore, $x^{i,k}$ exponentially converges to $x_\infty$ in expectation.
	\item [{\rm (ii)}] If parameters $\gamma$ and $\zeta$ in the SICBO algorithm satisfy \eqref{par1}, then there exists a positive random variable $\tilde{C}(\omega)$ such that
	\begin{equation}\label{asconvergence}
		\left\| x_\infty-x^{i,k}\right\|_1\leq d\tilde{C}(\omega) \frac{\mathrm{e}^{-k\alpha/2}}{1-\mathrm{e}^{-\alpha/2}},~\forall k\geq 0,~\mbox{a.s.}.
	\end{equation}
	where $x_\infty$ is defined as in Theorem \ref{th-common} and $\alpha$ can be found in Theorem \ref{th-consensus}. Therefore, $x^{i,k}$ exponentially converges to $x_\infty$ a.s..
\end{itemize}
\end{theorem}
\begin{proof}
(i) Using \eqref{eq-updatestate} recursively, for $\forall M>k$, one gets 
\begin{equation*}
x_l^{i,M}=x_l^{i,k}-\gamma\sum_{n=k}^{M-1}\left(x_l^{i,n}-\bar{x}_l^{\star,n}\right)-\sum_{n=k}^{M-1}\left(x_l^{i,n}-\bar{x}_l^{\star,n}\right)\eta_l^n.
\end{equation*}
Then, we have
\begin{equation}\label{M-k}
\begin{split}
	\left|x_l^{i,M}-x_l^{i,k}\right|&=\left|\sum_{n=k}^{M-1}\left(\gamma+\eta_l^n \right)\left(x_l^{i,n}-\bar{x}_l^{\star,n}\right) \right|\\
	&\leq\sum_{n=k}^{M-1}\left(\gamma+\left| \eta_l^n \right| \right)\max_{1\leq j\leq N}\left|x_l^{i,n}-x_l^{j,n}\right|.
\end{split}
\end{equation}
Taking expectation to \eqref{M-k} gives
\begin{equation*}
\begin{split}
	\mathbb{E}\left[\left|x_l^{i,M}-x_l^{i,k}\right|\right]&\leq\sum_{n=k}^{M-1}\left(\gamma+\mathbb{E}\left[\left| \eta_l^n\right| \right] \right)\mathbb{E}\left[\max_{1\leq j\leq N}\left|x_l^{i,n}-x_l^{j,n}\right|\right]\\
	&\leq(\gamma+\zeta)\sum_{n=k}^{M-1}(1-\gamma)^n\mathbb{E}\left[\max_{1\leq j\leq N}\left|x_l^{i,0}-x_l^{j,0}\right|\right]\\
	&=\frac{(\gamma+\zeta)(1-\gamma)^k\left(1-(1-\gamma)^{M-k} \right) }{\gamma}\mathbb{E}\left[\max_{1\leq j\leq N}\left|x_l^{i,0}-x_l^{j,0}\right|\right],
\end{split}
\end{equation*}
where the first inequality follows from the independence of $\eta_l^n$ and $x_l^{i,n}-x_l^{j,n}$ for any given $ n\geq 0$ and $l=1,\ldots,d$, the second inequality follows from \eqref{eta},  \eqref{i-j} and the fact that $\mathbb{E}[|X|]\leq \sqrt{\mathbb{E}[|X|^2]}$, for any random variable $X$.
Therefore, it holds 
\begin{equation}\label{norm1}
\begin{split}
&\mathbb{E}\left[\left\| x^{i,M}-x^{i,k}\right\|_1\right]\\\leq&\frac{(\gamma+\zeta)(1-\gamma)^k\left(1-(1-\gamma)^{M-k} \right)}{\gamma}\mathbb{E}\left[\max_{1\leq j\leq N}\left\|x^{i,0}-x^{j,0} \right\|_1 \right], ~\forall M>k\geq0.
\end{split}
\end{equation}
Letting $M\rightarrow\infty$ on the both sides of \eqref{norm1}, by Theorem \ref{th-common} and \eqref{par2}, the desired estimate \eqref{econvergence} is verified. In addition, we can draw the conclusion that $\{x^{i,k}\}_{i=1}^N$ exponentially converges to $x_\infty$ in expectation.

(ii) In view of the proof of Theorem \ref{th-consensus} (iii), we obtain that for $\forall i,j=1,\ldots,N$, $l=1,\ldots,d$, we can choose a small $\varepsilon>0$ such that
$$0<\alpha<\alpha+\varepsilon< 1-(1-\gamma)^2-\zeta^2,$$
then there exists a random variable $C_3(\omega):=\max_{0\leq k\leq K}\left\{\mathrm{e}^{(\alpha+\varepsilon)k}\left|x^{i,k}_l-x^{j,k}_l\right|^2\right\}$ satisfying
\begin{equation}\label{var}
\left| x^{i,k}_l-x^{j,k}_l\right|^2\leq C_3(\omega)\mathrm{e}^{-(\alpha+\varepsilon)k},\quad \forall k\geq 0,~\mbox{a.s.}.
\end{equation} 
Lemma \ref{lem-Y} yields that for $\varepsilon$ given in \eqref{var}, there is some random variable $C_2(\omega)$ such that
\begin{equation}\label{eta1}
\left|\eta_l^k\right|\leq C_2(\omega)e^{\varepsilon k/2}, \quad \forall k\geq0,~~\mbox{a.s.}.
\end{equation}
By \eqref{M-k}, it follows that for $M>k\geq 0$,
\begin{equation*}
\begin{split}
	\left|x_l^{i,M}-x_l^{i,k}\right|&\leq\sum_{n=k}^{M-1}\left(\gamma+C_2(\omega) \mathrm{e}^{\varepsilon n/2} \right) \max_{1\leq j\leq N}\left|x_l^{i,n}-x_l^{j,n} \right|\\
	&\leq\sum_{n=k}^{M-1}\left(\gamma+C_2(\omega) \mathrm{e}^{\varepsilon n/2} \right) \sqrt{C_3(\omega)}\mathrm{e}^{-n\left( \alpha+\varepsilon\right)/2}\\
	&=\sum_{n=k}^{M-1}\left( \gamma\mathrm{e}^{-n\varepsilon/2}+C_2(\omega)\right)\sqrt{C_3(\omega)}\mathrm{e}^{-n \alpha/2}\\
	&\leq \left(\gamma +C_2(\omega)\right)\sqrt{C_3(\omega)}\frac{\mathrm{e}^{-k\alpha/2} }{1-\mathrm{e}^{-\alpha/2}} ,~~\mbox{a.s.},
\end{split}
\end{equation*}
where the first inequality follows from \eqref{eta1}, the second inequality follows from \eqref{var} and the last inequality follows from $\alpha>0$.
Then, we have
\begin{equation}\label{norm1as}
\left\| x^{i,M}-x^{i,k}\right\|_1\leq d\tilde{C}(\omega) \frac{\mathrm{e}^{-k\alpha/2}}{1-\mathrm{e}^{-\alpha/2}},~\forall k\geq 0,~\mbox{a.s.} 
\end{equation}
with $\tilde{C}(\omega):=\left(\gamma +C_2(\omega)\right)\sqrt{C_3(\omega)}$.
Letting $M\rightarrow\infty$ in \eqref{norm1as} and using Theorem \ref{th-common}, we have the desired result in \eqref{asconvergence}. Therefore, $\alpha>0$ yields that $x^{i,k}$ exponentially converges to $x_\infty$ a.s..
\end{proof}
\subsection{Error estimate}
Since we have established the existence of the common consensus point of the SICBO
algorithm, we present an error analysis of the SICBO algorithm toward to the global minimum of \eqref{problem} in this subsection. Hereinafter, we will make use of the following notations and assumptions in terms of the smoothing function $\tilde{f}$, the parameters $\{\mu_k\}$ and the reference random variable $x_{in}$ in the SICBO algorithm.

By Definition \ref{def-sf}, note that for any given bounded set $\mathcal{X}\subseteq\mathbb{R}^d$, there exist positive parameters $\nu, \kappa, q$ such that $\tilde{f}$ satisfies
\begin{equation}\label{eq-L}
\sup\limits_{(x,\mu)\in\mathcal{X}\times(0,\bar{\mu}]}\left\|\nabla_{xx}^2\tilde{f}(x,\mu)\right\|\leq \nu\mu^{-q-1}
\end{equation}
and
\begin{equation}\label{eq-min}
\tilde{f}(x,\mu)\geq f(x)-\kappa\mu^{1-q}\geq f(x^*)-\kappa\bar{\mu}^{1-q}=:\tilde{f}_m, ~~\forall x\in\mathcal{X}, ~\mu\in(0,\bar{\mu}].
\end{equation}
\begin{assumption}\label{ass-par}
The parameters $\gamma$ and $\zeta$ in the SICBO algorithm satisfy 
\eqref{par1}.
\end{assumption}
\begin{assumption}\label{ass-ini}
The initial data $\{x^{i,0}\}_{i=1}^N$ is i.i.d and $x^{i,0}\sim x_{in}$, where $x_{in}$ is a reference random variable with a law which is absolutely continuous w.r.t the Lebesgue measure.
\end{assumption}
\begin{assumption}\label{ass-mu}
$\{\mu_k\}_{k\geq 0}$ is a positive, monotonically non-increasing sequence and satisfies that 
$$\lim_{k\rightarrow\infty}\mu_k=0 \quad\mbox{and}\quad \mu_0\leq\bar{\mu}.$$
Moreover, there exist constants $\theta$ and $\tau$ satisfying
\begin{equation*}\label{mu}
	\sum_{k=0}^{\infty}\mu_{k+1}^{-q}(\mu_k-\mu_{k+1})\leq \mu_0\theta \quad\mbox{and}\quad \sum_{k=0}^{\infty}\mathrm{e}^{-(2\gamma-\gamma^2-\zeta^2)k}\mu_k^{-1-q}\leq \tau.
\end{equation*}
\end{assumption}
%\begin{assumption}\label{ass-C3}
%	$C_3$ in Assumption \ref{ass-mu} satisfies that $\lim_{\mu_0\rightarrow0}C_3=0$.
%	\end{assumption}
%
%\begin{remark}
%	Assumption \ref{ass-mu} is trivial. For instance, the following cases satisfy this assumption.
%	\begin{itemize}
%		\item [{\rm (i)}] $\mu_k=\frac{\mu_0}{(1+k)^\sigma}$, where $\sigma$ satisfies that
%		\begin{equation*}
%			\sigma(1-q)>1 \quad\mbox{and}\quad0<\sigma(1+q)<-\gamma^2+2\gamma-\zeta^2;
%		\end{equation*}
%	\item [{\rm (i)}] $\mu_k=\sigma^k\mu_0$, where $\sigma$ satisfies that
%	\begin{equation*}
%		1\leq\sigma\mathrm{e}<\mathrm{e} \quad\mbox{and}\quad(1+q)<-\gamma^2+2\gamma-\zeta^2.
%	\end{equation*}
%	\end{itemize} 
%\end{remark}

Now, we are ready to provide the results concerning the error estimation between the objective function value of the consensus point and the global minimum of $f$.
\begin{theorem}\label{th-error}
Suppose Assumptions \ref{ass-par}-\ref{ass-mu} hold. If the parameters $\{\beta, \gamma, \zeta\}$, initial data $\mu_0$ and $\{x^{i,0}\}_{i=1}^N$ in the SICBO algorithm satisfy
\begin{equation}\label{condition}
\begin{split}
	&\left(\mathrm{e}^{-\beta\kappa\mu_0^{1-q}}-\epsilon\right)\mathbb{E}\left[\mathrm{e}^{-\beta f(x_{in})}\right]\\
	\geq& \mu_0\theta\beta\kappa\mathrm{e}^{-\beta\tilde{f}_m}
	+2\tau\nu\beta\sqrt{(\left(1+(1-\gamma)^2+\zeta^2\right)\left(\gamma^2+\zeta^2\right)}\mathrm{e}^{-\beta\tilde{f}_m}
	\sum_{l=1}^d\mathbb{E}\left[\mathcal{D}_l \right]
\end{split}
\end{equation}
for some $\epsilon\in\left(0,\mathrm{e}^{-\beta\kappa\mu_0^{1-q}}\right) $. Then, for the solution process $\{x^{i,k}\}_{i=1}^N$ to the SICBO algorithm, one has the following error estimate:
\begin{equation}\label{error}
	{\rm ess}\inf_{\omega\in\Omega}f(x_\infty(\omega))\leq f(x^\ast)+E(\beta),
\end{equation}
where $x_\infty$ is the consensus point defined as in Theorem \ref{th-common}, and $\lim_{\beta\rightarrow\infty}E(\beta)=0$.
\end{theorem}
\begin{proof}
Note that
\begin{equation*}
\begin{split}
&\frac{1}{N}\sum_{i=1}^N\mathrm{e}^{-\beta\tilde{f}( x^{i,k+1},\mu_{k+1}) }-\frac{1}{N}\sum_{i=1}^N\mathrm{e}^{-\beta\tilde{f}( x^{i,k},\mu_{k}) }\\
=&\frac{1}{N}\sum_{i=1}^N\mathrm{e}^{-\beta\tilde{f}( x^{i,k},\mu_{k}) }\left(\mathrm{e}^{-\beta\left(\tilde{f}( x^{i,k+1},\mu_{k+1}) -\tilde{f}( x^{i,k},\mu_{k}) \right)}-1\right)\\
\geq&\frac{1}{N}\sum_{i=1}^N\mathrm{e}^{-\beta\tilde{f}( x^{i,k},\mu_{k}) }(-\beta)\left(\tilde{f}( x^{i,k+1},\mu_{k+1}) -\tilde{f}( x^{i,k},\mu_{k}) \right)\\
%=&\frac{1}{N}\sum_{i=1}^N\mathrm{e}^{-\beta\tilde{f}( x^{i,k},\mu_{k}) }(-\beta)\left(\tilde{f}(x^{i,k+1},\mu_{k+1}) -\tilde{f}( x^{i,k+1},\mu_{k}) +\tilde{f}( x^{i,k+1},\mu_{k}) -\tilde{f}( x^{i,k},\mu_{k}) \right)\\
=&\left[-\frac{\beta}{N}\sum_{i=1}^N\mathrm{e}^{-\beta\tilde{f}( x^{i,k},\mu_{k}) }\left(\tilde{f}(x^{i,k+1},\mu_{k+1})-\tilde{f}(x^{i,k+1},\mu_{k})\right)\right]\\
&+\left[-\frac{\beta}{N}\sum_{i=1}^N\mathrm{e}^{-\beta\tilde{f}( x^{i,k},\mu_{k}) }\left(\tilde{f}(x^{i,k+1},\mu_{k})-\tilde{f}(x^{i,k},\mu_{k})\right)\right],
\end{split}
\end{equation*}
where the first inequality uses the fact $1+t\leq \mathrm{e}^t$, $\forall t\in\mathbb{R}$. Taking expectation on the both sides of the above inequality, we have
\begin{equation*}
\frac{1}{N}\sum_{i=1}^N\mathbb{E}\left[ \mathrm{e}^{-\beta\tilde{f}( x^{i,k+1},\mu_{k+1}) }\right] -\frac{1}{N}\sum_{i=1}^N\mathbb{E}\left[\mathrm{e}^{-\beta\tilde{f}( x^{i,k},\mu_{k}) }\right] =:\mathbb{E}\left[ \mathcal{I}_1\right] +\mathbb{E}\left[ \mathcal{I}_2\right] ,
\end{equation*}
where we denote
\begin{equation*}
\begin{split}
&\mathcal{I}_1=-\frac{\beta}{N}\sum_{i=1}^N\mathrm{e}^{-\beta\tilde{f}( x^{i,k},\mu_{k}) }\left(\tilde{f}(x^{i,k+1},\mu_{k+1})-\tilde{f}(x^{i,k+1},\mu_{k})\right),\\
&\mathcal{I}_2=-\frac{\beta}{N}\sum_{i=1}^N\mathrm{e}^{-\beta\tilde{f}( x^{i,k},\mu_{k}) }\left(\tilde{f}(x^{i,k+1},\mu_{k})-\tilde{f}(x^{i,k},\mu_{k})\right).
\end{split}
\end{equation*}
According to Theorem \ref{th-common} and \eqref{eq-weightedstate}, there exists a bounded set $\mathcal{X}$ such that for all $i=1,\ldots,N$,
\begin{equation}\label{X}
x^{i,k}\in\mathcal{X}\quad\mbox{and}\quad \bar{x}^{\star,k}\in\mathcal{X},~\forall k\geq 0,~{\rm a.s.}.
\end{equation}
By \eqref{eq-mu1} and \eqref{X} as well as the non-increasing property of $\left\lbrace \mu_k\right\rbrace $, we get
\begin{equation}\label{I1}
\begin{split}
\mathbb{E}\left[ \mathcal{I}_1\right] 
\geq-\frac{\beta\kappa}{N}\sum_{i=1}^N\mathrm{e}^{-\beta\tilde{f}_m}\mu_{k+1}^{-q}|\mu_{k+1}-\mu_k|
=-\beta\kappa \mathrm{e}^{-\beta\tilde{f}_m}\mu_{k+1}^{-q}(\mu_k-\mu_{k+1}).\\
\end{split}
\end{equation}
Recall the definition of $\bar{x}^{\star,k}$ in \eqref{eq-weightedstate}, and multiplying $\sum_{i=1}^N\mathrm{e}^{-\beta \tilde{f}(x^{i,k},\mu_k)}$ on the both sides of \eqref{eq-weightedstate} yields
$$\left(\sum_{i=1}^N\mathrm{e}^{-\beta \tilde{f}\left(x^{i,k},\mu_k\right)}\right)\bar{x}^{\star,k}=\sum_{i=1}^N\mathrm{e}^{-\beta \tilde{f}(x^{i,k},\mu_k)}x^{i,k}.$$
The $l$-th component of the above equation is as follows:
\begin{equation}\label{lth}
\left(\sum_{i=1}^N\mathrm{e}^{-\beta \tilde{f}(x^{i,k},\mu_k)}\right)\bar{x}^{\star,k}_l=\sum_{i=1}^N\mathrm{e}^{-\beta \tilde{f}(x^{i,k},\mu_k)}x^{i,k}_l,\quad \forall l=1,\ldots,d.
\end{equation}
Hence, we obtain
\begin{equation}\label{eq-0}
\begin{split}
&\sum_{i=1}^N\mathrm{e}^{-\beta \tilde{f}(x^{i,k},\mu_k)}\nabla_x\tilde{f}(\bar{x}^{\star,k},\mu_k)^{\mathrm{T}}\left(x^{i,k+1}-x^{i,k}\right)\\
%=&\sum_{i=1}^N\mathrm{e}^{-\beta \tilde{f}(x^{i,k},\mu_k)}\nabla_x\tilde{f}(\bar{x}^{\star,k},\mu_k)^{\mathrm{T}}\left(-\gamma\left(x^{i,k}-\bar{x}^{\star,k}\right)-\sum_{l=1}^d\left(x^{i,k}_l-\bar{x}^{\star,k}_l\right)\eta_l^ke_l\right)\\
=&\sum_{i=1}^N\mathrm{e}^{-\beta \tilde{f}(x^{i,k},\mu_k)}\nabla_x\tilde{f}(\bar{x}^{\star,k},\mu_k)^{\mathrm{T}}\left(-\sum_{l=1}^d\left( x^{i,k}_l-\bar{x}^{\star,k}_l\right) \left( \gamma+\eta_l^k\right) e_l\right)
=0,
\end{split}
\end{equation}
where the first equality follows from \eqref{eq-updatestate}, and the last equality follows from \eqref{lth}.
Next, we show the estimation of $\mathbb{E}\left[\mathcal{I}_2\right] $ as follows.
From the mean-value theorem with $c\in(0,1)$, we have
\begin{equation*}
	\begin{split}
		&\tilde{f}(x^{i,k+1},\mu_{k})-\tilde{f}(x^{i,k},\mu_{k})\\
		=&\nabla_x\tilde{f}(cx^{i,k+1}+(1-c)x^{i,k},\mu_k)^{\mathrm{T}}\left(x^{i,k+1}-x^{i,k}\right)\\
		=&\left( \nabla_x\tilde{f}(cx^{i,k+1}+(1-c)x^{i,k},\mu_k)-\nabla_x\tilde{f}(\bar{x}^{\star,k},\mu_k)\right)^{\mathrm{T}}\left(x^{i,k+1}-x^{i,k}\right)\\
		\geq&-\frac{ \nu}{\mu_k^{q+1}}\left\|cx^{i,k+1}+(1-c)x^{i,k}-\bar{x}^{\star,k}\right\|\left\|x^{i,k+1}-x^{i,k}\right\|\\
		=&-\frac{\nu}{\mu_k^{q+1}} \sqrt{\sum_{l=1}^d\left( x^{i,k}_l-\bar{x}^{\star,k}_l\right) ^2\left( 1-c\gamma-c\eta_l^k\right) ^2}\sqrt{\sum_{l=1}^d\left( x^{i,k}_l-\bar{x}^{\star,k}_l\right) ^2\left( \gamma+\eta_l^k\right) ^2},
		\end{split}
	\end{equation*}
where the second equality follows from \eqref{eq-0}, the first inequality follows from \eqref{eq-L} and \eqref{X}, and the third equality follows from \eqref{eq-updatestate}. We set
$$g(s):=\left( 1-s\gamma-s\eta_l^k \right) ^2,\quad 0\leq s\leq 1,$$
and observe that
$$\left( 1-s\gamma-s\eta_l^k\right) ^2\leq \max\{g(0),g(1)\}\leq g(0)+g(1)=1+\left(1-\gamma-\eta_l^k\right)^2.$$
Then, the above formula shows
\begin{equation}\label{ff}
	\small
	\begin{split}
&\tilde{f}(x^{i,k+1},\mu_{k})-\tilde{f}(x^{i,k},\mu_{k})\\
\geq& -\frac{ \nu}{\mu_k^{q+1}}\sqrt{\sum_{l=1}^d\left( x^{i,k}_l-\bar{x}^{\star,k}_l\right) ^2\left(1+\left( 1-\gamma-\eta_l^k\right) ^2\right)}\sqrt{\sum_{l=1}^d\left( x^{i,k}_l-\bar{x}^{\star,k}_l\right) ^2\left( \gamma+\eta_l^k\right) ^2} .
\end{split}
\end{equation}
%One combines \eqref{I1} and \eqref{I2} to derive
%\begin{equation}\label{I12}
%	\begin{split}
%		&\frac{1}{N}\sum_{i=1}^N\mathbb{E}\left[\mathrm{e}^{-\beta\tilde{f}(x^{i,k+1},\mu_{k+1})}\right] -\frac{1}{N}\sum_{i=1}^N\mathbb{E}\left[\mathrm{e}^{-\beta\tilde{f}(x^{i,k},\mu_{k})}\right] \\
%		\geq&-\beta\kappa \mathrm{e}^{-\beta\tilde{f}_m}\mu_{k+1}^{-q}(\mu_k-\mu_{k+1})\\
%		&-\frac{\beta \nu\mathrm{e}^{-\beta\tilde{f}_m}}{N\mu_k^{q+1}}\sum_{i=1}^N\mathbb{E}\left[\sqrt{\sum_{l=1}^d\left( x^{i,k}_l-\bar{x}^{\star,k}_l\right) ^2\left(1+\left( 1-\gamma-\eta_l^k\right) ^2\right)}\sqrt{\sum_{l=1}^d\left( x^{i,k}_l-\bar{x}^{\star,k}_l\right) ^2\left( \gamma+\eta_l^k\right) ^2}\right] .
%	\end{split}
%\end{equation}
By the Cauchy-Schwarz inequality
$\left|\mathbb{E}[XY]\right|^2\leq \mathbb{E}[X^2]\mathbb{E}[Y^2]$, we have
\begin{equation}\label{EE}
	\small
\begin{split}
&\mathbb{E}\left[\sqrt{\sum_{l=1}^d\left( x^{i,k}_l-\bar{x}^{\star,k}_l\right) ^2\left(1+\left( 1-\gamma-\eta_l^k\right) ^2\right)}\sqrt{\sum_{l=1}^d\left( x^{i,k}_l-\bar{x}^{\star,k}_l\right) ^2\left( \gamma+\eta_l^k\right) ^2}\right]\\
\leq&\sqrt{\mathbb{E}\left[\sum_{l=1}^d\left( x^{i,k}_l-\bar{x}^{\star,k}_l\right)^2
	\left(1+\left( 1-\gamma-\eta_l^k\right) ^2\right)\right]\mathbb{E}\left[\sum_{l=1}^d\left( x^{i,k}_l-\bar{x}^{\star,k}_l\right) ^2\left( \gamma+\eta_l^k\right) ^2\right]}\\
=&\sqrt{\left(\sum_{l=1}^d\mathbb{E}\left[ \left|x^{i,k}_l-\bar{x}^{\star,k}_l\right|^2\right] 
	\mathbb{E}\left[1+\left( 1-\gamma-\eta_l^k\right) ^2\right]\right)\left(\sum_{l=1}^d\mathbb{E}\left[ \left|x^{i,k}_l-\bar{x}^{\star,k}_l\right|^2\right] \mathbb{E}\left[ \left|\gamma+\eta_l^k\right|^2\right] \right)}\\
=&\sqrt{\left(\left(1+(1-\gamma)^2+\zeta^2\right)\mathbb{E}\left[ \left\|x^{i,k}-\bar{x}^{\star,k}\right\|^2\right] 
	\right)\left(\left(\gamma^2+\zeta^2\right)\mathbb{E}\left[ \left\|x^{i,k}-\bar{x}^{\star,k}\right\|^2\right] \right)}\\
%=&\sqrt{\left(1+(1-\gamma)^2+\zeta^2\right)\left(\gamma^2+\zeta^2\right)}\mathbb{E}\left[ \left\|x^{i,k}-\bar{x}^{\star,k}\right\|^2\right] \\
\leq&2\sqrt{\left(1+(1-\gamma)^2+\zeta^2\right)\left(\gamma^2+\zeta^2\right)}\mathrm{e}^{-(2\gamma-\gamma^2-\zeta^2)k}\sum_{l=1}^d\mathbb{E}\left[\mathcal{D}_l \right],
\end{split}
\end{equation}
where the first equality follows from the fact that $x^{i,k}_l-\bar{x}^{\star,k}_l$ is independent on $\eta_l^k$ or any given $ n\geq 0$ and $l=1,\ldots,d$, the second equality follows from \eqref{eta}, and the last inequality is a direct consequence of \eqref{eq-ex}. Consequently, \eqref{eq-min}, \eqref{ff} and \eqref{EE} yield that
\begin{equation}\label{I2}
	\mathbb{E}[\mathcal{I}_2]\geq-\frac{2\nu\beta\sqrt{\left(1+(1-\gamma)^2+\zeta^2\right)\left(\gamma^2+\zeta^2\right)}\mathrm{e}^{-\beta\tilde{f}_m}\mathrm{e}^{-(2\gamma-\gamma^2-\zeta^2)k}}{\mu_k^{q+1}}\sum_{l=1}^d\mathbb{E}\left[\mathcal{D}_l \right].
\end{equation}
Then by \eqref{I1} and \eqref{I2}, we find
\begin{equation}\label{EI12}
\begin{split}
&\frac{1}{N}\sum_{i=1}^N\mathbb{E}\left[ \mathrm{e}^{-\beta\tilde{f}(x^{i,k+1},\mu_{k+1})}\right] -\frac{1}{N}\sum_{i=1}^N\mathbb{E}\left[ \mathrm{e}^{-\beta\tilde{f}(x^{i,k},\mu_{k})}\right] \\
%\geq&-\beta\kappa \mathrm{e}^{-\beta\tilde{f}_m}\mu_{k+1}^{-q}(\mu_k-\mu_{k+1})\\
%&-\frac{\beta \nu\mathrm{e}^{-\beta\tilde{f}_m}}{N\mu_k^{q+1}}\sum_{i=1}^N\mathbb{E}\left[\sqrt{\sum_{l=1}^d\left( x^{i,k}_l-\bar{x}^{\star,k}_l\right) ^2\left(1+\left( 1-\gamma-\eta_l^k\right) ^2\right)}\sqrt{\sum_{l=1}^d\left( x^{i,k}_l-\bar{x}^{\star,k}_l\right) ^2\left( \gamma+\eta_l^k\right) ^2}\right]\\
\geq&-\beta\kappa \mathrm{e}^{-\beta\tilde{f}_m}\mu_{k+1}^{-q}(\mu_k-\mu_{k+1})\\
&-\frac{2\nu\beta\sqrt{\left(1+(1-\gamma)^2+\zeta^2\right)\left(\gamma^2+\zeta^2\right)}\mathrm{e}^{-\beta\tilde{f}_m}\mathrm{e}^{-(2\gamma-\gamma^2-\zeta^2)k}}{\mu_k^{q+1}}\sum_{l=1}^d\mathbb{E}\left[\mathcal{D}_l \right].
\end{split}
\end{equation}
From \eqref{eq-mu3}, we obtain
$$\mathbb{E}\left[ \mathrm{e}^{-\beta\tilde{f}(x^{i,0},\mu_{0})}\right] \geq\mathbb{E}\left[ \mathrm{e}^{-\beta\left(f(x^{i,0})+\kappa\mu_0^{1-q}\right)}\right] 
=\mathrm{e}^{-\beta\kappa\mu_0^{1-q}}\mathbb{E}\left[ \mathrm{e}^{-\beta f(x^{i,0})}\right] .$$
Then summing up \eqref{EI12} over $k$ and by the inequality above, one has
\begin{equation}\label{1}
	\small
\begin{split}
&\frac{1}{N}\sum_{i=1}^N\mathbb{E}\left[ \mathrm{e}^{-\beta\tilde{f}(x^{i,k},\mu_{k})}\right]\\ %\geq&\frac{1}{N}\sum_{i=1}^N\mathbb{E}\left[ \mathrm{e}^{-\beta\tilde{f}(x^{i,0},\mu_{0})}\right] -\beta\kappa\mathrm{e}^{-\beta\tilde{f}_m}\sum_{n=0}^{k-1}\mu_{n+1}^{-q}(\mu_n-\mu_{n+1})\\
%&-2\nu\beta\sqrt{(\left(1+(1-\gamma)^2+\zeta^2\right)\left(\gamma^2+\zeta^2\right)}\mathrm{e}^{-\beta\tilde{f}_m}\sum_{n=0}^{k-1}
%\left(\mathrm{e}^{-(2\gamma-\gamma^2-\zeta^2)n}\mu_n^{-1-q}\right)
%\sum_{l=1}^d\mathbb{E}\left[\mathcal{D}_l \right]\\
\geq& \frac{\mathrm{e}^{-\beta\kappa\mu_0^{1-q}}}{N}\sum_{i=1}^N\mathbb{E}\left[ \mathrm{e}^{-\beta f(x^{i,0})}\right]-\beta\kappa\mathrm{e}^{-\beta\tilde{f}_m}\sum_{n=0}^{k-1}\mu_{n+1}^{-q}(\mu_n-\mu_{n+1})\\
&-2\nu\beta\sqrt{(\left(1+(1-\gamma)^2+\zeta^2\right)\left(\gamma^2+\zeta^2\right)}\mathrm{e}^{-\beta\tilde{f}_m}\sum_{n=0}^{k-1}
\left(\mathrm{e}^{-(2\gamma-\gamma^2-\zeta^2)n}\mu_n^{-1-q}\right)
\sum_{l=1}^d\mathbb{E}\left[\mathcal{D}_l \right].
\end{split}
\end{equation}
Following from Definition \ref{def-sf}-(ii), Theorem \ref{th-common} and the Dominated Convergence Theorem \cite{ResnickA1999}, we have
\begin{equation}\label{limit}
\lim_{k\rightarrow\infty}\frac{1}{N}\sum_{i=1}^N\mathbb{E}\left[\mathrm{e}^{-\beta\tilde{f}(x^{i,k},\mu_k)}\right]=\mathbb{E}\left[\mathrm{e}^{-\beta f(x_\infty)}\right].
\end{equation}
Then letting $k\rightarrow\infty$ in \eqref{1}, we get
\begin{equation*}
\begin{split}
\mathbb{E}\left[ \mathrm{e}^{-\beta f(x_\infty)}\right] \geq& \mathrm{e}^{-\beta\kappa\mu_0^{1-q}}\mathbb{E}\left[ \mathrm{e}^{-\beta f(x_{in})}\right] -\mu_0\theta\beta\kappa\mathrm{e}^{-\beta\tilde{f}_m}\\
&-2\tau\nu\beta\sqrt{(\left(1+(1-\gamma)^2+\zeta^2\right)\left(\gamma^2+\zeta^2\right)}\mathrm{e}^{-\beta\tilde{f}_m}
\sum_{l=1}^d\mathbb{E}\left[\mathcal{D}_l \right]\\
\geq&\epsilon\mathbb{E}\left[ \mathrm{e}^{-\beta f(x_{in})}\right],
\end{split}
\end{equation*}
where the first inequality follows from \eqref{limit} and Assumptions \ref{ass-ini} and \ref{ass-mu}, the second inequality follows from \eqref{condition}.
Therefore, we have
\begin{equation}\label{log}
\mathrm{e}^{-\beta{\rm ess}\inf\limits_{\omega\in\Omega}f(x_\infty(\omega))}=\mathbb{E}\left[ \mathrm{e}^{-\beta{\rm ess}\inf\limits_{\omega\in\Omega}f(x_\infty(\omega))}\right] \geq\mathbb{E}\left[ \mathrm{e}^{-\beta f(x_\infty)}\right] \geq\epsilon\mathbb{E}\left[ \mathrm{e}^{-\beta f(x_{in})}\right] .
\end{equation}
Taking logarithm and dividing by $-\beta$ to the both sides of \eqref{log} to find
\begin{equation*}
\begin{split}
{\rm ess}\inf_{\omega\in\Omega}f(x_\infty(\omega))\leq& -\frac{1}{\beta}\log\left(\epsilon\mathbb{E}\left[ \mathrm{e}^{-\beta f(x_{in})}\right] \right)\\
=&-\frac{1}{\beta}\log\mathbb{E}\left[ \mathrm{e}^{-\beta f(x_{in})}\right] -\frac{1}{\beta}\log\epsilon.
\end{split}
\end{equation*}
%Now Proposition \ref{Laplace} implies that
%$$\lim_{\beta\rightarrow\infty}-\frac{1}{\beta}\log\mathbb{E}\left[ \mathrm{e}^{-\beta f(x_{in})}\right] =f(x^*).$$
So by Proposition \ref{Laplace}, if we define
\begin{equation*}\label{Ebetaa}
E(\beta):=-\frac{1}{\beta}\log\mathbb{E}\left[ \mathrm{e}^{-\beta f(x_{in})}\right] -f(x^*)-\frac{1}{\beta}\log\epsilon,
\end{equation*}
then
$${\rm ess}\inf_{\omega\in\Omega}f(x_\infty(\omega))\leq f(x^\ast)+E(\beta),$$
where $\lim_{\beta\rightarrow\infty}E(\beta)=0$.
\end{proof}
\begin{remark}
%	We provide two comments on \eqref{condition} in Theorem \ref{th-error}. 
%	
%	(i) By Remark \ref{re-sf}, we note that if $f$ is locally Lipschitz continuous, then $q=0$ and \eqref{condition} can be rewritten as
%	\begin{equation*}
%	\begin{split}
	%		\left(\mathrm{e}^{-\beta\kappa\mu_0}-\epsilon\right)\mathbb{E}\left[\mathrm{e}^{-\beta f(x_{in})}\right]\geq
	%		&\beta\kappa\mu_0 \mathrm{e}^{-\beta\tilde{f}_m}+\frac{2\nu\beta\mu_0\sqrt{(1+(1-\gamma)^2+\zeta^2)(\gamma^2+\zeta^2)}\mathrm{e}^{-\beta \tilde{f}_m}}{1-\mathrm{e}^{-b-2\gamma+\gamma^2+\zeta^2}}\\
	%		&\sum\limits_{l=1}^d\mathbb{E}\left[\max\limits_{1\leq i\leq N}\left( x_l^{i,0}-\bar{x}_l^{(0)}\right) ^2\right];
	%	\end{split}
%\end{equation*}
%	If $f=|g(x)|^p$, then $q=1-p$ and \eqref{condition} can be rewritten as
%	\begin{equation*}
	%		\begin{split}
		%			\left(\mathrm{e}^{-\beta\kappa\mu_0^{p}}-\epsilon\right)\mathbb{E}\left[\mathrm{e}^{-\beta f(x_{in})}\right]\geq
		%			&\frac{\beta\kappa\mu_0^{p} (\mathrm{e}^{-b}-1)\mathrm{e}^{bq}\mathrm{e}^{-\beta\tilde{f}_m}}{1-\mathrm{e}^{bq}}+\frac{2\nu\beta\mu_0^{p-2}\sqrt{(1+(1-\gamma)^2+\zeta^2)(\gamma^2+\zeta^2)}\mathrm{e}^{-\beta \tilde{f}_m}}{1-\mathrm{e}^{-b(2-p)-2\gamma+\gamma^2+\zeta^2}}\\
		%			&\sum\limits_{l=1}^d\mathbb{E}\left[\max\limits_{1\leq i\leq N}\left( x_l^{i,0}-\bar{x}_l^{(0)}\right) ^2\right],
		%		\end{split}
	%	\end{equation*}

%(ii) 
By multiplying $\mathrm{e}^{\beta\tilde{f}_m}$ on the both sides of \eqref{condition}, we have
\begin{equation}\label{condition2}
	\begin{split}
		&\left(\mathrm{e}^{-\beta\kappa\mu_0^{1-q}}-\epsilon\right)\mathbb{E}\left[\mathrm{e}^{\beta (\tilde{f}_m-f(x_{in}))}\right]\\
		\geq&\theta\mu_0\beta\kappa+2\tau\nu\beta\sqrt{(1+(1-\gamma)^2+\zeta^2)(\gamma^2+\zeta^2)}
		\sum\limits_{l=1}^d\mathbb{E}\left[\mathcal{D}_l \right].
	\end{split}
\end{equation}

If let $\beta\rightarrow0$ in \eqref{condition2}, we can see that the left side of \eqref{condition2} converges to $1-\varepsilon$, and the right side converges to $0$.
This implies that $\beta>0$ satisfying \eqref{condition} is guaranteed to exist. However, if let $\beta\rightarrow \infty$, due to the fact that $\tilde{f}_m-f(x_{in})<0$, we obtain that the left side of \eqref{condition2} is not larger than $0$ while the right side of it tends to $\infty$. Therefore, the estimate in Theorem \ref{th-error} is only an error estimate for the SICBO algorithm with a proper $\beta$, but can not guarantee a sufficient small error estimate on the objective function value to $f_{min}$.

In what follows, we give a sufficient condition to guarantee the small enough of the error $E(\beta)$.
It illustrates that if the parameters are properly selected, the objective function value at the global consensus point can be close enough to the global minimum. It is noteworthy that no previous literature about discrete CBO algorithm, to our knowledge, has investigated this thing.
\end{remark}

\begin{corollary}\label{co-err}
Under Assumptions \ref{ass-par}-\ref{ass-mu}, for any given $\delta>0$, there exist $\epsilon$, parameters $\{\mu_0, \beta, \gamma, \zeta\}$ and a suitable $x_{in}$ in Assumption \ref{ass-ini} such that $E(\beta)\leq \delta$, where $E(\beta)$ is defined in \eqref{error}.
%satisfying \eqref{condition}. Moreover, with these parameters, the SICBO algorithm has
%\begin{equation}\label{essinf}
%	{\rm ess}\inf_{\omega\in\Omega}f(x_\infty(\omega))\leq f(x^\ast)+E(\beta),
%\end{equation}
%where $x_\infty$ is defined as in Theorem \ref{th-common} and $E(\beta)$ defined in \eqref{Ebetaa} satisfies $E(\beta)\leq \delta$.
\end{corollary}
\begin{proof}
For any given $\delta>0$, let 
\begin{equation}\label{epsilon}
\epsilon= \frac{\mathrm{e}^{-2\mu_0^{1-q}\beta\kappa-\beta\delta}}{\mu_0\theta\beta\kappa+\mathrm{e}^{-\mu_0^{1-q}\beta\kappa-\beta\delta}}.
\end{equation}
Hereinafter, we will show that there exist corresponding parameters $\{\mu_0, \beta, \gamma, \zeta\}$ satisfying \eqref{condition}.

Firstly, we claim that for the given $\delta>0$ and $\beta>0$, there exists a $\mu_0>0$ satisfying
\begin{equation}\label{re3}
\delta>-\frac{1}{\beta}\log\left( 1-\mu_0\theta\beta\kappa\mathrm{e}^{\mu_0^{1-q}\beta\kappa}\right),
\end{equation}
in which if we let $\mu_0\rightarrow 0$ in \eqref{re3}, then it is clear that the right side of it converges to $0$ and the left side equals $\delta>0$. Upon rearranging the terms in \eqref{re3}, one has

\begin{equation}\label{re2}
\mathrm{e}^{-\mu_0^{1-q}\beta\kappa}(1-\mathrm{e}^{-\beta\delta})-\mu_0\theta\beta\kappa>0.
\end{equation}
Next, in view of
\begin{equation*}
\mathrm{e}^{-\mu_0^{1-q}\beta\kappa}-\epsilon-\mu_0\theta\beta\kappa=\frac{\mu_0\theta\beta\kappa\left(\mathrm{e}^{-\mu_0^{1-q}\beta\kappa}(1-\mathrm{e}^{-\beta\delta})-\mu_0\theta\beta\kappa \right) }{\mu_0\theta\beta\kappa+\mathrm{e}^{-\mu_0^{1-q}\beta\kappa-\beta\delta}}
\end{equation*}
and \eqref{re2}, then for the given $\delta>0$ and $\beta>0$, there exists a sufficiently small $\mu_0>0$ such that
\begin{equation}\label{re1}
\mathrm{e}^{-\mu_0^{1-q}\beta\kappa}-\epsilon-\mu_0\theta\beta\kappa>0.
\end{equation} 
By \eqref{re1}, it is easy to find a sufficiently small $\beta>0$ and a suitable $x_{in}$ such that
\begin{equation*}\label{re4}
\left(\mathrm{e}^{-\mu_0^{1-q}\beta\kappa}-\epsilon\right)\mathbb{E}\left[\mathrm{e}^{\beta\left(\tilde{f}_{m}- f(x_{in})\right)}\right]-\mu_0\theta\beta\kappa>0.
\end{equation*}
Therefore, there are appropriate parameters $\{\gamma, \zeta\}$ and initial data such that
\begin{equation*}
\begin{split}
	&\sqrt{(1+(1-\gamma)^2+\zeta^2)(\gamma^2+\zeta^2)}\sum\limits_{l=1}^d\mathbb{E}\left[\mathcal{D}_l \right]\\
	\leq& \frac{\left(\mathrm{e}^{-\mu_0^{1-q}\beta\kappa}-\epsilon\right)\mathbb{E}\left[\mathrm{e}^{\beta\left(\tilde{f}_{m}- f(x_{in})\right)}\right]-\mu_0\theta\beta\kappa}{2\tau\nu\beta},
\end{split}
\end{equation*}
which suggests \eqref{condition2} as well as \eqref{condition}.
To sum up, we confirm that for any $\delta>0$, there exist suitable parameters $\{\epsilon, \mu_0, \beta, \gamma, \zeta\}$ guarantee \eqref{condition}. 

%In addition, from Theorem \ref{th-error}, we have
%$${\rm ess}\inf f(x_\infty)\leq f(x^*)+E(\beta),$$
%where $E(\beta)$ is defined in \eqref{Ebetaa}.
For convenience in the forthcoming analysis, we denote
$$\mathcal{L}:=\frac{\left(\mu_0\theta\kappa+2\tau\nu\sqrt{(1+(1-\gamma)^2+\zeta^2)(\gamma^2+\zeta^2)}
\sum\limits_{l=1}^d\mathbb{E}\left[\mathcal{D}_l \right]\right)\beta\mathrm{e}^{-\beta\tilde{f}_{m}}}{\mathrm{e}^{-\mu_0^{1-q}\beta\kappa}-\epsilon},$$
where $\{\epsilon, \beta, \mu_0, \gamma, \zeta\}$  is the aforementioned parameters that make \eqref{condition} valid.
From Theorem \ref{th-error}, \eqref{condition} implies  \eqref{error} and
\begin{equation}\label{Ebeta}
E(\beta)\leq-\frac{1}{\beta}\log\mathcal{L}-f(x^*)-\frac{1}{\beta}\log\epsilon.
\end{equation}
Moreover, by \eqref{epsilon}, one has
$$\frac{\mu_0\theta\beta\kappa\epsilon}{\mathrm{e}^{-\mu_0^{1-q}\beta\kappa}-\epsilon}= \mathrm{e}^{-\beta(\delta+\kappa\mu_0^{1-q})}.$$
Then, we have
\begin{equation}\label{delta1}
\frac{\left(\mu_0\theta\kappa+2\tau\nu\sqrt{(1+(1-\gamma)^2+\zeta^2)(\gamma^2+\zeta^2)}\sum\limits_{l=1}^d\mathbb{E}\left[\mathcal{D}_l \right]\right)\beta\epsilon}{\mathrm{e}^{-\mu_0^{1-q}\beta\kappa}-\epsilon}\geq\mathrm{e}^{-\beta(\delta+\kappa\mu_0^{1-q})}.
\end{equation}
We multiply the both sides of \eqref{delta1} by $\mathrm{e}^{-\beta\tilde{f}_{m}}$ to obtain
\begin{equation}\label{delta2}
\mathcal{L}\epsilon\geq\mathrm{e}^{-\beta(\delta+\kappa\mu_0^{1-q}+\tilde{f}_{m})}\geq\mathrm{e}^{-\beta(\delta+f(x^*))},
\end{equation}
where the last inequality follows from the definition of $\tilde{f}_{m}$ in \eqref{eq-min} and Assumption \ref{ass-mu}. %the assumption that $\mu_0\leq\bar{\mu}$.
%The definition of $\tilde{f}_{m}$ in \eqref{eq-min} along with the fact that $\mu_0\leq\bar{\mu}$ yield
%\begin{equation}\label{delta2}
%	\mathcal{L}\epsilon\geq\mathrm{e}^{-\beta(\delta+f(x^*))}.
%\end{equation}
Taking logarithm and dividing by $-\beta$ to the both sides of \eqref{delta2}, we have
\begin{equation*}
-\frac{1}{\beta}\log\left(\mathcal{L}\epsilon\right) \leq\delta+f(x^*),
\end{equation*}
i.e.
\begin{equation}\label{delta3}
-\frac{1}{\beta}\log\mathcal{L}-f(x^*)-\frac{1}{\beta}\log\epsilon\leq\delta.
\end{equation}
Therefore, combining \eqref{Ebeta} and \eqref{delta3}, we obtain the desired result that $E(\beta)\leq \delta$.
\end{proof}

\section {Numerical experiments}\label{sec6}
In order to validate the theoretical results and investigate the numerical behaviors of the SICBO algorithm, we illustrate some experimental results in this section. All experiments are performed in PYTHON 3.8 on a Lenovo PC (2.60 GHz, 32.0 GB RAM). Prior to the numerical experiments, we give some discussion on the implementation of the SICBO algorithm.

\begin{itemize}
\item[{\rm(i)}] Unless otherwise specified, the stop criterion is set as the one mentioned in Remark \ref{re-stop} (iii) with nonnegative constants $\varepsilon_1=\varepsilon_2=1\mathrm{e}$-$10$.
\item[{\rm(ii)}]  The initial point $\{x^{{i,0}}\}_{i=1}^N$ is randomly initialized from a uniform distribution in a given set $\mathcal{S}$.
%\item[{\rm(iii)}] Recall the calculation of $\bar{x}^\star$ in each iteration from \eqref{eq-weightedstate}, for a large $\beta$, the quantity of $\mathrm{e}^{-\beta \tilde{f}(x^{i,k},\mu_k)}$ could be extremely small. Therefore, we multiply $\max_{1\leq i \leq N}\mathrm{e}^{-\beta \tilde{f}(x^{i,k},\mu_k)}$ on both the numerator and denominator to avoid the numerical underflow.
\end{itemize}
\begin{example}\label{ex-1}
In this example, we verify the global consistency of the proposed SICBO algorithm by considering the following problem
\begin{equation}\label{ex-f1}
	\min_{x\in\mathbb{R}^d} f(x)=\frac{1}{d}\sum_{l=1}^d\left[|x_l|-10\cos(2\pi x_l)+10\right],
\end{equation}
where the objective function $f$ is nonsmooth and nonconvex, and the global minimizer of it is $x^*=(0,\ldots,0)^\mathrm{T}\in\mathbb{R}^d$. It has a lot of local minimizers, whose number grows exponentially fast in term of the dimensionality. 
Using \eqref{smoothing}, we construct the smoothing function of $f$ as follows:
\begin{equation*}
	\tilde{f}(x,\mu)=\frac{1}{d}\sum_{l=1}^d\left[\varphi_1(x_l,\mu)-10\cos(2\pi x_l)+10\right],
\end{equation*}
and take the smoothing parameter by $\mu_k=\frac{1}{(1+k)^2}$. We use the SICBO algorithm to solve problem \eqref{ex-f1} with $d=4$. We set $\mathcal{S}=[-5,5]^d$ as the domain for initial data. The fixed parameters in the SICBO algorithm are given as follows:
\begin{equation*}
	N=100,\quad \beta=100, \quad\zeta=0.1,\quad\gamma=0.01.
\end{equation*} 
From Fig. \ref{fig:BB}, we see that the solution processes of the $100$ particles in the SICBO algorithm ultimately converges to a common consensus point $x_\infty$, where $x_\infty=(-0.0139, -0.0445, -0.0006, 0.0029)^\mathrm{T}$. This result supports the emergence of global consensus of the SICBO algorithm.
\begin{figure}[h]
	\centering
	\includegraphics[width=0.95\linewidth]{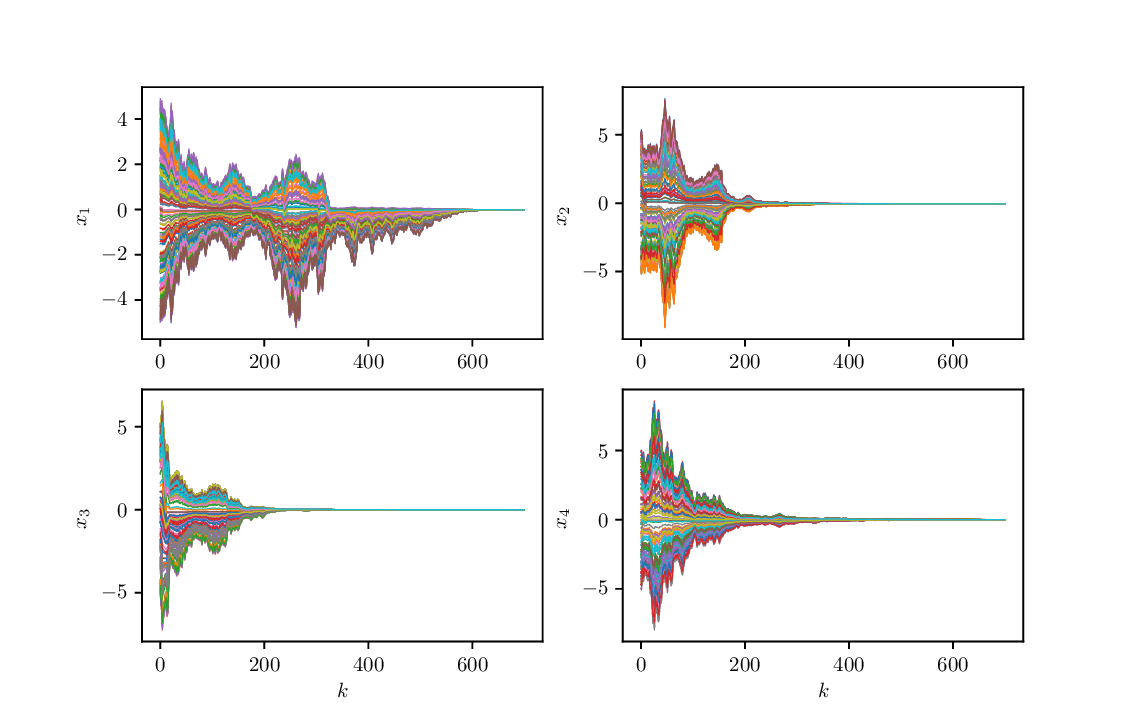}
	\caption{Trajectories of the $100$ particles for Example \ref{ex-1} with $d=4$.}
	\label{fig:BB}
\end{figure}
\end{example}
\begin{example}\label{ex-2}
We compare the numerical performance of the SICBO algorithm and the  Stochastic Subgradient Descent (SSD) method in solving the global minimizer of nonsmooth and nonconvex optimization problem, and show the great superiority of our proposed algorithm by this example. We construct the following objective function $f$:
\begin{equation}\label{ex-f2}
	\min_{x\in\mathbb{R}} f(x)=\frac{1}{n}\sum_{i=1}^n g(x,\hat{x}_i),
\end{equation}
where
\begin{equation*}
	g(x,\hat{x}_i)=\mathrm{e}^{\sin\left(2\left(x+\frac{\pi}{2}\right)^2\right)}+\frac{1}{10}\left| x-\hat{x}_i\right|,\quad \hat{x}_i\sim\mathcal{N}(0,0.1).
\end{equation*}
\begin{figure}[h]
	\centering
	\includegraphics[width=0.6\linewidth]{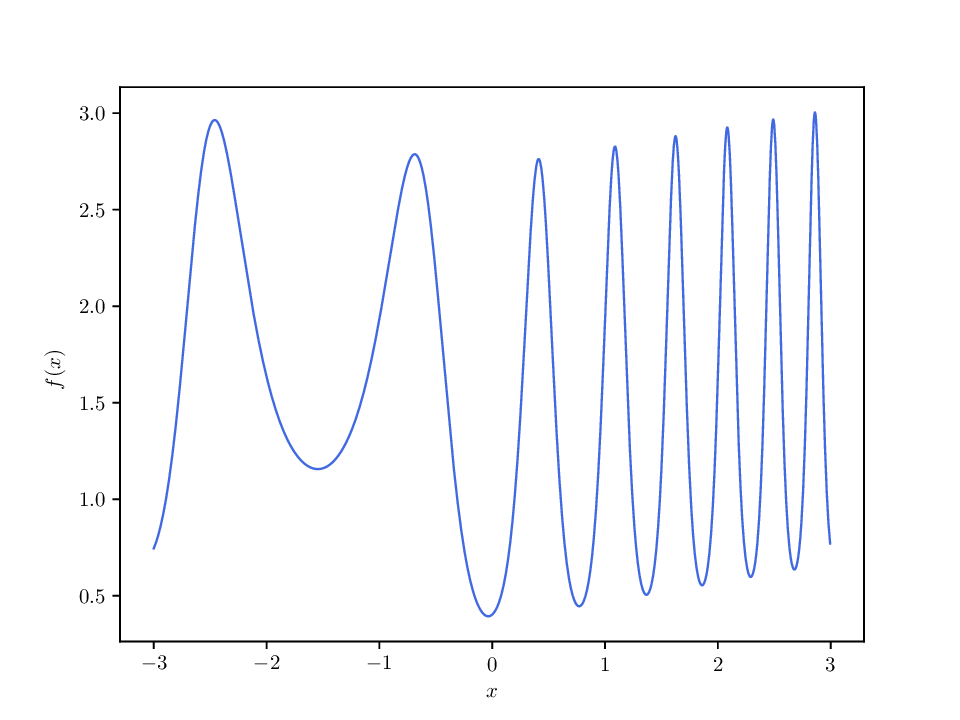}
	\caption{The objective function $f$ in \eqref{ex-f2} with $n=10^4$.}
	\label{fig:funpic}
\end{figure}
Fig. \ref{fig:funpic} shows that $f$ in \eqref{ex-f2} attains multiple local minimizers. And the global minimizer to the optimization problem \eqref{ex-f2} is $x^*\approx -0.0352$, with the global minima $f(x^*)\approx 0.3933$.

SSD algorithm updates $x_k$ in the following way,
\begin{equation*}
	x^{k+1}=x^k-\frac{s}{|b^k|}\sum_{i\in b^k}\partial_{x}g(x^k,\hat{x}_i),
\end{equation*}
where $b^k$ is an index set randomly drawn from $\{1,\ldots,n\}$, $|b^k|$ is the cardinality of set $b^k$ and $s$ is the given step length. We set $n=10^4$, $|b^k|=20$ and $s=0.01$.
In addition, for our method, the smoothing function is also constructed by using \eqref{smoothing} and set $\mu_k=\frac{1}{(1+k)^2}$. The parameters are set as below:
\begin{equation*}
	N=100,\quad \beta=50, \quad\zeta=0.1,\quad\gamma=0.01.
\end{equation*} 
For both methods, we choose the initial points randomly in the same way with $\mathcal{S}=[-3,3]$ and separately run 100 simulations. One simulation is considered as a successful test, if $\max_{1\leq i\leq N}\left\| x^{i,k}-x^*\right\| <5\mathrm{e}$-$03$ for the SICBO algorithm and $\left\| x^k-x^*\right\| <5\mathrm{e}$-$03$ for the SSD algorithm.

In Table \ref{table1}, we list the success rates (\texttt{suc-rat}), solution errors (\texttt{sol-err}) defined by $\frac{1}{N}\sum_{i=1}^N\left\| x^{i,k}-x^*\right\|$ (SICBO) or $\left\| x^k-x^*\right\|$ (SSD), and objective function value errors (\texttt{fun-err}) calculated by $\frac{1}{N}\sum_{i=1}^N \left| f(x^{i,k})-f(x^*)\right|$ (SICBO) or $\left| f(x^k)-f(x^*)\right|$ (SSD), which are averaged over 100 runs. It is easily seen that compared with the SSD method, even better results are achieved by the SICBO algorithm. The SSD method always cannot escape from the local minimum and find the global minimum in most cases, however, the SICBO method can do the both things with high probabilities.
\begin{table}[htbp]
	\caption{Numerical results of the SSD method and the SICBO algorithm.} \label{table1}
	\begin{center}
		\begin{tabular}{cccc}
			\hline
			Algorithm&\texttt{suc-rat}  &\texttt{sol-err}  &\texttt{fun-err}  \\
			\hline
			SSD&$21\%$  &$1.41$  &$0.29$  \\
			\hline
			SICBO&$\textbf{92\%}$  &$\bf{2.69}\mathrm{\bf{e}}$\bf{-}$\bf{03}$  &$\bf{1.66}\mathrm{\bf{e}}$\bf{-}$\bf{04}$  \\
			\hline
		\end{tabular}
	\end{center}
\end{table}
\end{example}	
\begin{example}\label{ex-3}
In this example, we investigate the influence of the number of particles $N$ and the exponential weight $\beta$ for the performance of the SICBO algorithm for solving the global minimizer of \eqref{problem}. Eight considered objective functions with the dimension of $d=3$ are listed in Table \ref{function}. Note that these functions have many local minimizers and a unique global minimizer $x^*=(0,\ldots,0)^{\mathrm{T}}\in\mathbb{R}^d$ with the optimal function value $0$.
\begin{table}[h!]
	\caption{Considered test functions in Example \ref{ex-3}}
	\centering
	\resizebox{0.8\textwidth}{!}{
		\begin{tabular}{c||c} 
			\hline
			& Objective function $f$ in problem \eqref{problem} \\
			\hline 
			$f_1$ & $\frac{1}{d}\sum_{l=1}^d[|x_l|-10\cos(2\pi x_l)+10]$ \\ [2ex]
			$f_2$ & $-10\exp\left( -0.2\sqrt{\frac{1}{d}\sum_{l=1}^d |x_l|}\right)-\exp\left(\frac{1}{d}\sum_{l=1}^d\cos(2\pi x_l) \right)+10+\exp(1)$ \\ [2ex]
			$f_3$ & $\left[\sum_{l=1}^d \sin^2(x_l)-\exp\left( -\sum_{l=1}^d x_l^2\right)\right]\exp\left( -\sum_{l=1}^d\sin^2(\sqrt{|x_l|})\right)+1$ \\ [2ex]
			$f_4$ & $\frac{1}{4000}\sum_{l=1}^d|x_l|-\prod_{l=1}^d\cos\left(\frac{x_l}{\sqrt{l}}\right)+1$ \\ [2ex]
			$f_5$ & $\sum_{l=1}^d|x_l|+\prod_{l=1}^d|x_l|$ \\ [2ex]
			$f_6$ & $10\left\|x\sin(10x)-0.1x\right\|_1$\\ [2ex]
			$f_7$ & $1-\prod_{l=1}^d\cos(x_l)\exp(-|x_l|)$\\ [2ex]
			$f_8$ & $1-\cos\left(2\pi\sqrt{\sum_{l=1}^dx_l^2} \right)+0.1\sqrt{\|x\|_1} $\\ [2ex]
			\hline
	\end{tabular}}
	\label{function}
\end{table}

The smoothing functions of these test functions are constructed by \eqref{smoothing} and the following parameters are used for the simulations
\begin{equation*}
	\mathcal{S}=[-3,3]^d,\quad\zeta=0.1,\quad\gamma=0.01,\quad \mu_k=\frac{1}{(1+k)^2}.
\end{equation*} 
And we consider a run successful when $\max_{1\leq i\leq N}\left\| x^{i,k}-x^*\right\| <1\mathrm{e}$-$02$. Tables \ref{table3} and \ref{table4} report the success rates (\texttt{suc-rat}), solution errors (\texttt{sol-err}) defined by $\frac{1}{N}\sum_{i=1}^N\left\| x^{i,k}-x^*\right\|^2$, and objective function value errors (\texttt{fun-err}) calculated by $\frac{1}{N}\sum_{i=1}^N \left| f(x^{i,k})-f(x^*)\right|$, averaged over 100 runs.  In detail, Table \ref{table3} shows the numerical results obtained by the SICBO algorithm with $\beta=1\mathrm{e}\mathrm{+}15$ when different $N$ are given. Table \ref{table4} presents the performance by the SICBO algorithm with $N=200$ for different values of $\beta$. The above results indicate that increasing $N$ and $\beta$ can effectively improve the performance of the proposed algorithm, enabling it to accurately find the global minimizer of nonconvex and nonsmooth objective functions.
%\begin{table}[htbp]
%	\caption{The numerical results of both the SSD method and the SCBO algorithm.} \label{table2}
%	\begin{center}
	%		\begin{tabular}{ccccc}
		%			\hline
		%			function&~&$d=2$&$d=3$&$d=4$  \\
		%			\hline
		%			$f_1$&\texttt{suc-rat}&$100$&$80$&$80$\\
		%			~&\texttt{sol-err}&$4.13\mathrm{e}$-$06$&$1.99\mathrm{e}$-$01$&$9.93\mathrm{e}$-$01$\\
		%			~&\texttt{fun-err}&$1.17\mathrm{e}$-$03$&$6.98\mathrm{e}$-$02$&$2.44\mathrm{e}$-$01$\\
		%			\hline
		%			$f_2$&\texttt{suc-rat}&$100$&$80$&$80$\\
		%			~&\texttt{sol-err}&$4.13\mathrm{e}$-$06$&$1.99\mathrm{e}$-$01$&$9.93\mathrm{e}$-$01$\\
		%			~&\texttt{fun-err}&$1.17\mathrm{e}$-$03$&$6.98\mathrm{e}$-$02$&$2.44\mathrm{e}$-$01$\\
		%			\hline
		%			$f_3$&\texttt{suc-rat}&$100$&$80$&$80$\\
		%			~&\texttt{sol-err}&$4.13\mathrm{e}$-$06$&$1.99\mathrm{e}$-$01$&$9.93\mathrm{e}$-$01$\\
		%			~&\texttt{fun-err}&$1.17\mathrm{e}$-$03$&$6.98\mathrm{e}$-$02$&$2.44\mathrm{e}$-$01$\\
		%			\hline
		%			$f_4$&\texttt{suc-rat}&$100$&$80$&$80$\\
		%			~&\texttt{sol-err}&$4.13\mathrm{e}$-$06$&$1.99\mathrm{e}$-$01$&$9.93\mathrm{e}$-$01$\\
		%			~&\texttt{fun-err}&$1.17\mathrm{e}$-$03$&$6.98\mathrm{e}$-$02$&$2.44\mathrm{e}$-$01$\\
		%			\hline
		%			$f_5$&\texttt{suc-rat}&$100$&$80$&$80$\\
		%			~&\texttt{sol-err}&$4.13\mathrm{e}$-$06$&$1.99\mathrm{e}$-$01$&$9.93\mathrm{e}$-$01$\\
		%			~&\texttt{fun-err}&$1.17\mathrm{e}$-$03$&$6.98\mathrm{e}$-$02$&$2.44\mathrm{e}$-$01$\\
		%			\hline
		%		\end{tabular}
	%	\end{center}
%\end{table}
\begin{table}[h!]
	\caption{Results of the SICBO algorithm with different values of $N$ for Example \ref{ex-3}.} \label{table3}
	\begin{center}
		\resizebox{0.77\textwidth}{!}{
			\begin{tabular}{cccccc}
				\hline
				function&~&$N=50$&$N=100$&$N=200$&$N=400$\\
				\hline
				$f_1$&\texttt{suc-rat}&$38\%$&$75\%$&$93\%$&$\textbf{97\%}$\\
				~&\texttt{sol-err}&$4.92\mathrm{e}$-$01$&$2.00\mathrm{e}$-$01$&$\bf{4.31}\mathrm{\bf{e}}$-$\bf{05}$&$1.66\mathrm{e}$-$04$\\
				~&\texttt{fun-err}&$1.94\mathrm{e}$-$01$&$8.62\mathrm{e}$-$02$&$\bf{4.58}\mathrm{\bf{e}}$-$\bf{03}$&$1.26\mathrm{e}$-$02$\\
				\hline
				$f_2$&\texttt{suc-rat}&$47\%$&$56\%$&$79\%$&$\textbf{100\%}$\\
				~&\texttt{sol-err}&$3.98\mathrm{e}$-$03$&$9.94\mathrm{e}$-$02$&$5.13\mathrm{e}$-$05$&$\bf{8.85}\mathrm{\bf{e}}$-$\bf{06}$\\
				~&\texttt{fun-err}&$2.80\mathrm{e}$-$01$&$2.54\mathrm{e}$-$01$&$5.57\mathrm{e}$-$02$&$\bf{3.69}\mathrm{\bf{e}}$-$\bf{02}$\\
				\hline
				$f_3$&\texttt{suc-rat}&$34\%$&$75\%$&$93\%$&$\textbf{98\%}$\\
				~&\texttt{sol-err}&$4.80\mathrm{e}$-$03$&$2.72\mathrm{e}$-$04$&$\bf{3.34}\mathrm{\bf{e}}$-$\bf{05}$&$5.27\mathrm{e}$-$05$\\
				~&\texttt{fun-err}&$5.52\mathrm{e}$-$02$&$9.80\mathrm{e}$-$03$&$\bf{1.27}\mathrm{\bf{e}}$-$\bf{03}$&$3.70\mathrm{e}$-$03$\\
				\hline
				$f_4$&\texttt{suc-rat}&$51\%$&$75\%$&$91\%$&$\textbf{100\%}$\\
				~&\texttt{sol-err}&$3.52\mathrm{e}$-$03$&$7.28\mathrm{e}$-$04$&$9.85\mathrm{e}$-$05$&$\bf{1.75}\mathrm{\bf{e}}$-$\bf{05}$\\
				~&\texttt{fun-err}&$6.96\mathrm{e}$-$04$&$3.57\mathrm{e}$-$04$&$2.20\mathrm{e}$-$05$&$\bf{3.38}\mathrm{\bf{e}}$-$\bf{06}$\\
				\hline
				$f_5$&\texttt{suc-rat}&$52\%$&$68\%$&$87\%$&$\textbf{98\%}$\\
				~&\texttt{sol-err}&$9.17\mathrm{e}$-$03$&$9.54\mathrm{e}$-$04$&$1.07\mathrm{e}$-$04$&$\bf{2.94}\mathrm{\bf{e}}$-$\bf{06}$\\
				~&\texttt{fun-err}&$6.17\mathrm{e}$-$02$&$2.09\mathrm{e}$-$02$&$6.31\mathrm{e}$-$03$&$\bf{3.50}\mathrm{\bf{e}}$-$\bf{03}$\\
				\hline
				$f_6$&\texttt{suc-rat}&$55\%$&$66\%$&$86\%$&$\textbf{91\%}$\\
				~&\texttt{sol-err}&$2.78\mathrm{e}$-$03$&$7,36\mathrm{e}$-$04$&$7.67\mathrm{e}$-$05$&$\bf{3.14}\mathrm{\bf{e}}$-$\bf{05}$\\
				~&\texttt{fun-err}&$2.26\mathrm{e}$-$01$&$8.23\mathrm{e}$-$02$&$9.69\mathrm{e}$-$03$&$\bf{6.84}\mathrm{\bf{e}}$-$\bf{03}$\\
				\hline
				$f_7$&\texttt{suc-rat}&$50\%$&$69\%$&$93\%$&$\textbf{98\%}$\\
				~&\texttt{sol-err}&$4.03\mathrm{e}$-$02$&$5.07\mathrm{e}$-$03$&$1.37\mathrm{e}$-$05$&$\bf{4.99}\mathrm{\bf{e}}$-$\bf{06}$\\
				~&\texttt{fun-err}&$3.28\mathrm{e}$-$02$&$4.93\mathrm{e}$-$02$&$2.14\mathrm{e}$-$03$&$\bf{1.34}\mathrm{\bf{e}}$-$\bf{03}$\\
				\hline
				$f_6$&\texttt{suc-rat}&$33\%$&$42\%$&$73\%$&$\textbf{89\%}$\\
				~&\texttt{sol-err}&$3.98\mathrm{e}$-$01$&$5.97\mathrm{e}$-$01$&$2.99\mathrm{e}$-$01$&$\bf{1.24}\mathrm{\bf{e}}$-$\bf{04}$\\
				~&\texttt{fun-err}&$6.00\mathrm{e}$-$02$&$6.38\mathrm{e}$-$02$&$3.35\mathrm{e}$-$02$&$\bf{8.22}\mathrm{\bf{e}}$-$\bf{03}$\\
				\hline
		\end{tabular}}
	\end{center}
\end{table}
\begin{table}[h!]
	\caption{Results of the SICBO algorithm with different values of $\beta$ for Example \ref{ex-3}.} \label{table4}
	\begin{center}
		\resizebox{0.8\textwidth}{!}{
			\begin{tabular}{cccccc}
				\hline
				function&~&$\beta=1\mathrm{e}$+$08$&$\beta=1\mathrm{e}$+$12$&$\beta=1\mathrm{e}$+$16$&$\beta=1\mathrm{e}$+$20$\\
				\hline
				$f_1$&\texttt{suc-rat}&$48\%$&$65\%$&$81\%$&$\textbf{99\%}$\\
				~&\texttt{sol-err}&$3.03\mathrm{e}$-$01$&$1.99\mathrm{e}$-$01$&$1.97\mathrm{e}$-$01$&$\bf{1.59}\mathrm{\bf{e}}$\bf{-}$\bf{05}$\\
				~&\texttt{fun-err}&$1.27\mathrm{e}$-$01$&$8.67\mathrm{e}$-$02$&$6.83\mathrm{e}$-$02$&$\bf{2.14}\mathrm{\bf{e}}$\bf{-}$\bf{03}$\\
				\hline
				$f_2$&\texttt{suc-rat}&$76\%$&$73\%$&$83\%$&$\textbf{97\%}$\\
				~&\texttt{sol-err}&$3.68\mathrm{e}$-$04$&$7.06\mathrm{e}$-$04$&$1.78\mathrm{e}$-$04$&$\bf{6.85}\mathrm{\bf{e}}$\bf{-}$\bf{06}$\\
				~&\texttt{fun-err}&$1.04\mathrm{e}$-$01$&$1.10\mathrm{e}$-$01$&$8.54\mathrm{e}$-$02$&$\bf{3.66}\mathrm{\bf{e}}$\bf{-}$\bf{02}$\\
				\hline
				$f_3$&\texttt{suc-rat}&$69\%$&$80\%$&$94\%$&$\textbf{99\%}$\\
				~&\texttt{sol-err}&$1.56\mathrm{e}$-$04$&$1.65\mathrm{e}$-$04$&$\bf{1.87}\mathrm{\bf{e}}$\bf{-}$\bf{0}5$&$2.80\mathrm{e}$-$05$\\
				~&\texttt{fun-err}&$8.17\mathrm{e}$-$03$&$7.16\mathrm{e}$-$03$&$3.99\mathrm{e}$-$03$&$\bf{3.01}\mathrm{\bf{e}}$\bf{-}$\bf{03}$\\
				\hline
				$f_4$&\texttt{suc-rat}&$69\%$&$78\%$&$\textbf{100\%}$&$\textbf{100\%}$\\
				~&\texttt{sol-err}&$2.64\mathrm{e}$-$03$&$7.01\mathrm{e}$-$05$&$\bf{2.85}\mathrm{\bf{e}}$\bf{-}$\bf{06}$&$1.58\mathrm{e}$-$05$\\
				~&\texttt{fun-err}&$6.62\mathrm{e}$-$04$&$1.36\mathrm{e}$-$05$&$\bf{1.29}\mathrm{\bf{e}}$\bf{-}$\bf{06}$&$4.78\mathrm{e}$-$06$\\
				\hline
				$f_5$&\texttt{suc-rat}&$66\%$&$62\%$&$71\%$&$\textbf{96\%}$\\
				~&\texttt{sol-err}&$4.16\mathrm{e}$-$04$&$9.53\mathrm{e}$-$03$&$1.19\mathrm{e}$-$03$&$\bf{6.77}\mathrm{\bf{e}}$\bf{-}$\bf{05}$\\
				~&\texttt{fun-err}&$1.33\mathrm{e}$-$02$&$5.21\mathrm{e}$-$02$&$1.54\mathrm{e}$-$02$&$\bf{4.17}\mathrm{\bf{e}}$\bf{-}$\bf{03}$\\
				\hline
				$f_6$&\texttt{suc-rat}&$62\%$&$62\%$&$85\%$&$\textbf{89\%}$\\
				~&\texttt{sol-err}&$2.04\mathrm{e}$-$03$&$1.10\mathrm{e}$-$03$&$4.28\mathrm{e}$-$04$&$\bf{1.53}\mathrm{\bf{e}}$\bf{-}$\bf{05}$\\
				~&\texttt{fun-err}&$1.68\mathrm{e}$-$01$&$1.15\mathrm{e}$-$01$&$3.52\mathrm{e}$-$02$&$\bf{4.53}\mathrm{\bf{e}}$\bf{-}$\bf{03}$\\
				\hline
				$f_7$&\texttt{suc-rat}&$66\%$&$76\%$&$71\%$&$\textbf{86\%}$\\
				~&\texttt{sol-err}&$2.27\mathrm{e}$-$03$&$4.24\mathrm{e}$-$04$&$4.51\mathrm{e}$-$03$&$\bf{2.29}\mathrm{\bf{e}}$\bf{-}$\bf{04}$\\
				~&\texttt{fun-err}&$3.09\mathrm{e}$-$02$&$1.22\mathrm{e}$-$02$&$4.97\mathrm{e}$-$02$&$\bf{1.12}\mathrm{\bf{e}}$\bf{-}$\bf{02}$\\
				\hline
				$f_6$&\texttt{suc-rat}&$55\%$&$75\%$&$67\%$&$\textbf{87\%}$\\
				~&\texttt{sol-err}&$3.98\mathrm{e}$-$01$&$2.19\mathrm{e}$-$01$&$3.97\mathrm{e}$-$01$&$\bf{1.98}\mathrm{\bf{e}}$\bf{-}$\bf{01}$\\
				~&\texttt{fun-err}&$4.47\mathrm{e}$-$02$&$2.61\mathrm{e}$-$02$&$4.15\mathrm{e}$-$02$&$\bf{2.55}\mathrm{\bf{e}}$\bf{-}$\bf{02}$\\
				\hline
		\end{tabular}}
	\end{center}
\end{table}
\end{example}

\begin{example}\label{ex-4}
In this example, we evaluate the numerical behaviors of the SICBO algorithm for training a class of deep neural networks, which can be modeled by
\begin{equation}\label{dnn}
	\min_{W,b} \frac{1}{M}\sum_{m=1}^M\left\| \sigma(W_L\sigma(\cdots\sigma(W_1u_m+b_1)+\cdots)+b_L)-y_m\right\|^2, 
\end{equation}
where $\left\lbrace u_m\in\mathbb{R}^{N_0}\right\rbrace_{m=1}^M$ and $\left\lbrace y_m\in\mathbb{R}^{N_L}\right\rbrace_{m=1}^M$ are the given data, variables $\{W_\ell\in\mathbb{R}^{N_\ell\times N_{\ell-1}}\}_{\ell=1}^L$ represent the weight matrices, $\{b_\ell\in\mathbb{R}^{N_\ell}\}_{\ell=1}^L$  represent the bias vectors, and $\sigma(\cdot)$ stands for ReLU activation function, i.e. $\sigma(s):=\max\{0,s\}$.

We randomly generate some test problems in the following way. Denote $M$ and $M_{test}$ the numbers of training and test samplings, respectively. The output data $\left\lbrace y_m\in\mathbb{R}^{N_L}\right\rbrace_{m=1}^{M+M_{test}}$ are generated from an $L$-layer neural network with the similar way to the method proposed in \cite{WeiAn2023}, i.e.,
\begin{equation*}
	y_m=\sigma(W_L\sigma(\cdots\sigma(W_1u_m+b_1)+b_2\cdots)+b_L)+\tilde{y}_m,
\end{equation*}
where $u_m\sim\mathcal{N}(\xi,\Sigma^\mathrm{T}\Sigma)$, $\xi={\rm{randn}}(N_0,1)$, $\Sigma={\rm{randn}}(N_0,1)$ and $\tilde{y}_m=0.05*\rm{randn}(1,1)$. Here ${\rm{randn}}(n,m)$ denotes an $n\times m$ randomly generated matrix which follows the standard Gaussian distribution. And we set $M=80$, $M_{test}=20$, $L=4$, $N_0=5$, $N_1=N_2=N_3=10$ and $N_4=1$.

We introduce the following two measurements (training error \texttt{Errtrain} and test error \texttt{Errtest}), to evaluate the performance of our algorithm:
\begin{equation*}
	\mbox{\texttt{Errtrain}}=\frac{1}{M}\sum_{m=1}^M \left\| \sigma(W_L\sigma(\cdots\sigma(W_1u_m+b_1)+b_2\cdots)+b_L)-y_m\right\|^2, 
\end{equation*}
\begin{equation*}
	\mbox{\texttt{Errtest}}=\frac{1}{M_{test}}\sum_{m=M+1}^{M+M_{test}} \left\| \sigma(W_L\sigma(\cdots\sigma(W_1u_m+b_1)+b_2\cdots)+b_L)-y_m\right\|^2.
\end{equation*}

For the nonsmooth function in \eqref{dnn}, we choose 
\begin{equation*}
	\varphi_2(s,\mu)=
	\begin{cases}
		\displaystyle \max\{0,s\}, &\mbox{if}~|s|\geq\mu/2\\
		\displaystyle \frac{s^2}{2\mu}+\frac{s}{2}+\frac{\mu}{8}, &\mbox{if}~|s|<\mu/2\\
	\end{cases}
\end{equation*}
as a smoothing function of $\sigma(s)$ to get $\tilde{f}$ in the SICBO algorithm. And we also randomly initialize the variables $(W,b)$ with $
\mathcal{S}=[-1,1]^d$, $d=291$ and set the parameters as
\begin{equation*}
	N=200,\quad \beta=1\mathrm{e}\mbox{\rm{+}}20, \quad\zeta=0.1,\quad\gamma=0.01,\quad \mu_k=\mathrm{e}^{0.1k},\quad \varepsilon_1=\varepsilon_2=1\mathrm{e}\mbox{\rm{-}}06.
\end{equation*}
Table \ref{table5} shows the numerical results of the SICBO algorithm for training the four-layer neural network with the ReLU activation function (averaged over $50$ simulations), including the \texttt{Errtrain}, \texttt{Errtest} and number of iterations. Fig. \ref{fig:errpic} displays the averaged \texttt{Errtrain} and \texttt{Errtest} of $200$ particles with the x-axis varying on iterations in one of the $50$ simulations. These results suggest that the SICBO algorithm can effectively reduce the training and test error for problem \eqref{dnn} to a satisfactory level.
\begin{table}[h!]
	\caption{Numerical results of the SICBO algorithm averaged over  $50$ simulations for Example \ref{ex-4}.} \label{table5}
	\begin{center}
		\begin{tabular}{ccc}
			\hline
			\texttt{Errtrain}  &\texttt{Errtest}  &\texttt{Iteration}  \\
			\hline
			$3.64\mathrm{e}$-$03$  &$3.66\mathrm{e}$-$03$  &$482$  \\
			\hline
		\end{tabular}
	\end{center}
\end{table}
\begin{figure}[h]
	\centering
	\includegraphics[width=0.7\linewidth]{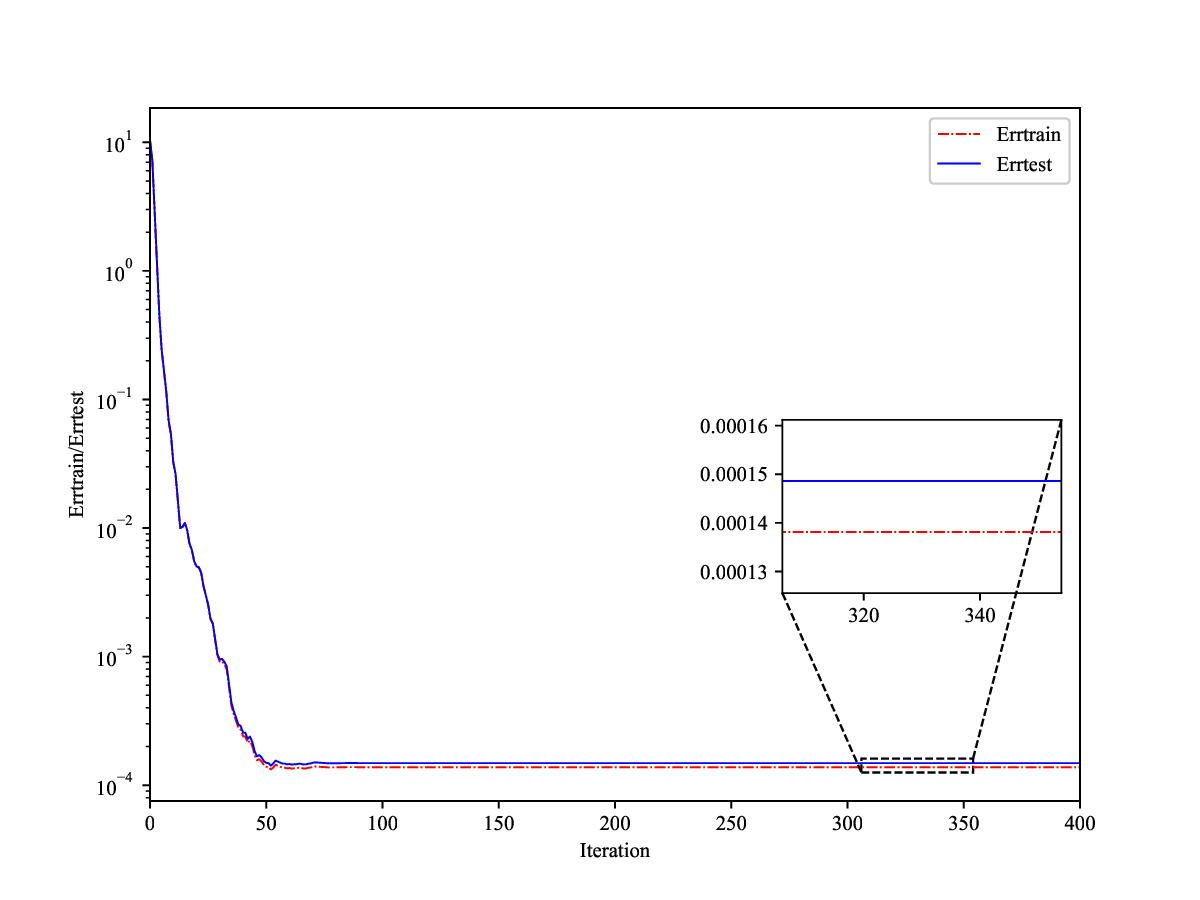}
	\caption{Averaged training error and test error in one of the $50$ simulations.}
	\label{fig:errpic}
\end{figure}
\end{example}
\section {Conclusions}\label{sec7}
We have proposed the SICBO algorithm for solving the global minimizer of the general unconstrained nonconvex, continuous but not necessary Lipschitz continuous optimization problems. A rigorous convergence analysis has been provided without using corresponding mean-field model. This algorithm utilizes smoothing method to overcome the nonsmoothness of objective function, and combines with the consensus-based optimization method to achieve the global optimization. 
It has been proved that the iterative process of the SICBO algorithm can establish global consensus. And it exponentially converges to the common consensus point in expectation and almost surely. 
We have further given a more detailed error estimate to show that the objective function value at the consensus point is close enough to the global minimum. Some numerical experiments have been presented to illustrate the theoretical results and demonstrate the appreciable performance of the proposed SICBO algorithm, including the reported results on some classes of test functions and a class of deep neural network training problems.

\begin{acknowledgements}
 The research of this author is partially supported by the National Natural Science Foundation of China grants (12425115, 12271127, 62176073).
%The research is supported by the National Natural Science Foundation of China grants (No. 12271127, 62176073), the National Key Research and Development Program of China (No. 2021YFA1003500) and the Fundamental Research Funds for the Central Universities (No. 2022FRFK060017).
\end{acknowledgements}

\section*{Data Availability}
The datasets generated during the current study are available from the corresponding author on reasonable request.

 \section*{Declarations}

 \textbf{Conflict of interest} The authors have not disclosed any competing interests.

\bibliographystyle{spmpsci}      % mathematics and physical sciences
\bibliography{cbosmoothing}   % name your BibTeX data base

\end{document}